%% file: Main.tex
\documentclass[a4paper,UKenglish,cleveref, autoref, thm-restate]{lipics-v2021}

\input{preamble}

\title{Connected $k$-partition of $k$-connected graphs and $c$-claw-free graphs} 

\titlerunning{Connected $k$-part. of $k$-conn. \& $c$-claw-free graphs} 

\author{Ralf Borndörfer}{Zuse Institute Berlin, Berlin}{borndoerfer@zib.de}{https://orcid.org/0000-0001-7223-9174}{}
\author{Katrin Casel}{Hasso Plattner Institute, University of Potsdam, Germany}{Katrin.Casel@hpi.de}{https://orcid.org/0000-0001-6146-8684}{}
\author{Davis Issac}{Hasso Plattner Institute, University of Potsdam, Germany}{Davis.Issac@hpi.de}{https://orcid.org/0000-0001-5559-7471}{}
\author{Aikaterini Niklanovits}{Hasso Plattner Institute, University of Potsdam, Germany}{Aikaterini.Niklanovits@hpi.de}{https://orcid.org/0000-0002-4911-4493}{}
\author{Stephan Schwartz}{Zuse Institute Berlin, Berlin}{schwartz@zib.de}{https://orcid.org/0000-0003-2901-5065}{}
\author{Ziena Zeif}{Hasso Plattner Institute, University of Potsdam, Germany}{Ziena.Zeif@hpi.de}{https://orcid.org/0000-0003-0378-1458}{}

\authorrunning{R. Borndörfer,K. Casel, D. Issac A. Niklanovits, S. Schwartz and Z. Zeif} 

\Copyright{Ralf Borndörfer, Katrin Casel, Davis Issac, Aikaterini Niklanovits, Stephan Schwartz and Ziena Zeif}



\category{} 

\relatedversion{} 

\nolinenumbers 

\hideLIPIcs  

\EventEditors{John Q. Open and Joan R. Access}
\EventNoEds{2}
\EventLongTitle{42nd Conference on Very Important Topics (CVIT 2016)}
\EventShortTitle{CVIT 2016}
\EventAcronym{CVIT}
\EventYear{2016}
\EventDate{December 24--27, 2016}
\EventLocation{Little Whinging, United Kingdom}
\EventLogo{}
\SeriesVolume{42}
\ArticleNo{23}

\begin{document}

	\maketitle
	
	\begin{abstract}
	
		A \emph{connected partition} is a partition of the vertices of a graph into sets that induce connected subgraphs. 
		Such partitions naturally occur in many application areas such as road networks, and image processing.
		In these settings, it is often desirable to partition into a fixed number of parts of roughly of the same size or weight.
		The resulting computational problem  is called Balanced Connected Partition (BCP).
		The two classical objectives for BCP are to maximize the weight of the smallest, or minimize the weight of the largest component.	We study BCP on $c$-claw-free graphs, the class of graphs that do not have $K_{1,c}$ as an induced subgraph,
		and present efficient $(c-1)$-approximation algorithms for both objectives.
		In particular, for $3$-claw-free graphs, also simply known as claw-free graphs, we obtain a $2$-approximation. 
		Due to the claw-freeness of line graphs, this also implies a $2$-approximation for the edge-partition version of BCP in general graphs.

		A harder connected partition problem arises from demanding a connected partition into $k$ parts that have (possibly) heterogeneous target weights $w_1,\ldots,w_k$. 
		In the 1970s Gy\H{o}ri and Lov\'{a}sz showed that if $G$ is $k$-connected and the target weights sum to the total size of $G$, such a partition exists. However, to this day no polynomial algorithm to compute such partitions exists for $k>4$.
		Towards finding such a partition $T_1,\ldots, T_k$ in $k$-connected graphs for general $k$, we show how to efficiently compute connected partitions that at least approximately meet the target weights, subject to the mild assumption that each $w_i$ is greater than the weight of the heaviest vertex. In particular, we give a 3-approximation for both the lower and the upper bounded version i.e.~we guarantee that each $T_i$ has weight at least $\frac{w_i}{3}$ or that each $T_i$ has weight most $3w_i$, respectively. Also, we present a both-side bounded version that produces a connected partition where each $T_i$ has size at least $\frac{w_i}{3}$ and at most $\max(\{r,3\}) w_i$, where $r \geq 1$ is the ratio between the largest and smallest value in $w_1, \dots, w_k$.
		In particular for the balanced version, i.e.~$w_1=w_2=, \dots,=w_k$, this gives a partition with $\frac{1}{3}w_i \leq w(T_i) \leq 3w_i$.  
	\end{abstract}
	
	\section{Introduction}
	\input{Intro}

	\section{Preliminaries}
	\input{prelims}
	
	\section{Approximation for BCP on $c$-claw-free graphs}\label{section::BCP}

	In this section we give an idea of how to prove Theorems~\ref{thm::MinMax} and~\ref{thm::MaxMin} by giving a $(c-1)$-approximation for {\maxminbcp} and {\minmaxbcp} on $K_{1,c}$-free graphs, full proofs of these results can be found in Appendix~\ref{app:BCP}. 
	We assume $c\ge 3$ as $c\le 2$ gives trivial graph classes.
	We first show that a connected partition for $K_{1,c}$-free graphs with parts of size in $[\lambda,(c-1)\lambda)$ for some fixed $\lambda$ can be found in linear time.
	For {\maxminbcp}, this algorithm has to be called many times while doing a binary search for the optimum value.
	We point out that it is not difficult to adapt these algorithms to unconnected graphs achieving the same approximation results.

	Exploiting that each vertex in any DFS-tree of a $K_{1,c}$-free graph has at most $c-1$ children, we can carefully extract connected components of a fixed size while also maintaining a DFS-tree for the remaining graph. Also, this can be done very efficiently, as stated in the following result.
\begin{lemma}
	\label{lemma::boundedSubgraph}
	Given a $K_{1,c}$-free graph $G$ and a DFS-tree of $G$. For any $w(G) \geq \lambda\ge \wmax$, there is an algorithm that finds a connected vertex set $S$ such that $\lambda \leq w(S) < (c-1) \lambda$ and $G-S$ is connected, in $\mathcal O(|V|)$ time.
	Furthermore, the algorithm finds a DFS-tree of $G - S$. 
\end{lemma}
	We use  {\algBalancedPartition} to denote the algorithm that exhaustively applies Lemma~\ref{lemma::boundedSubgraph}. Observe that {\algBalancedPartition} produces a connected partition $S_1,\dots,S_m$ where $w(S_i) \in [\lambda,(c-1)\lambda)$ for every $i \in [m-1]$ and $w(S_m) < (c-1)\lambda$ in linear time, where the achieved runtime follows by saving already processed subtrees.

	\input{SubAlgorithmEdgePartition}

	\section{Approximation of the Gy\H{o}ri-Lov\'{a}sz Theorem for \texorpdfstring{$k$}{k}-connected Graphs}
	\input{GyoeriLovasz}

	\bibliography{literature}
	\appendix
	\input{GL-appendix}
	\input{GLDouble-app.tex}
\end{document}

%% file: preamble.tex
\usepackage[utf8]{inputenc}

\usepackage{amsmath,amssymb,amsthm}
\usepackage{enumitem}

\usepackage{setspace}

\usepackage{tikz}
\usetikzlibrary{decorations.pathreplacing}
\usetikzlibrary{positioning,shapes,fit,calc}
\usepackage{graphicx}
\usepackage{subcaption} 
\usepackage[table]{colortbl}
\usepackage{nicefrac}
\usepackage{todonotes}
\usepackage{mathtools} 
\usepackage{multirow}
\usepackage{xcolor}
\usepackage{mathrsfs}
\usepackage{framed}

\usepackage[autostyle]{csquotes}
\MakeOuterQuote{"}

\usepackage[linesnumbered, 
algoruled, 
noend] 
{algorithm2e}

\usepackage{xcolor}  
\usepackage{hyperref}

\def\algmaxmin{\texttt{MaxMinApx}}

\def\algBoundedGL{\texttt{BoundedGL}}
\def\algDoubleBoundedGL{\texttt{DoubleBoundedGL}}
\def\algBalancedPartition{\texttt{BalancedPartition}}
\def\algTransfer{\texttt{TransferVertices}}

\def\maxminbcp{\textsc{Max-Min BCP}}
\def\minmaxbcp{\textsc{Min-Max BCP}}

\def\bigO{\mathcal{O}}

\def\O{\mathcal{O}}
\def\calS{\mathcal{S}}

\def\calN{\mathcal{N}}
\def\calU{\mathcal{U}}

\def\calV{\mathcal{V}}

\def\nu{\calN^{\bar{g}}}

\def\calQ{\mathcal{Q}}

\def\T{\mathcal{T}}
\def\wmax{w_{\max}}
\def\vmax{v_{\max}}

\def\bcp{BCP{$_k$}}

\def\cvp{\textsc{CVP}}

\def\og{\overline{G}}

\def\UpBoundGL{\max\{r,3\}}

\def\gyori{Gy\H{o}ri}
\def\lovasz{Lov\'{a}sz}

\newenvironment{tightcenter}
 {\parskip=0pt\par\nopagebreak\centering}
 {\par\noindent\ignorespacesafterend}

\usepackage{ctable}
\newlength{\RoundedBoxWidth}
\newsavebox{\GrayRoundedBox}
\newenvironment{GrayBox}[1]%
{\setlength{\RoundedBoxWidth}{\linewidth-4.5ex}
\def\boxheading{#1}
\begin{lrbox}{\GrayRoundedBox}
\begin{minipage}{\RoundedBoxWidth}%
}{%
\end{minipage}
\end{lrbox}%
\begin{tightcenter}%
\begin{tikzpicture}%
\node(Text)[draw=black!20,fill=white,rounded corners,%
inner sep=2ex,text width=\RoundedBoxWidth]%
{\usebox{\GrayRoundedBox}};
\coordinate(x) at (current bounding box.north west);
\node [draw=white,rectangle,inner sep=3pt,anchor=north west,fill=white]
at ($(x)+(10.5pt,.75em)$) {\boxheading};
\end{tikzpicture}
\end{tightcenter}\vspace{0pt}%
\ignorespacesafterend
}

\newenvironment{problem}[1]{\noindent\ignorespaces%
\FrameSep=6pt%
\parindent=0pt%
\vspace*{.5em}
\begin{GrayBox}{\textsc{#1}}%
\begin{tabular*}{\columnwidth}{@{\hspace{.25em}} >{\itshape} p{1.1cm} p{0.8\columnwidth} @{}}%
}{
\end{tabular*}%
\end{GrayBox}%
\vspace*{-.5em}
\ignorespacesafterend
}

%% file: Intro.tex
%

Partitioning a graph into connected subgraphs is a problem that arises in many application areas such as parallel processing, road network decomposition, image processing, districting problems, and robotics \cite{lucertini1993most,mohring2007partitioning,bulucc2016recent,vulnerability,zhou2019balanced}.
Often in these applications, it is required to find a partition into a specified number $k$ of connected subgraphs. 
For instance, in the parallel processing applications, the number of processors is restricted, and in robotics applications, the number of robots available is restricted. 
Formally, we call a partition $T_1,T_2,\cdots,T_k$ of the vertex set of graph, a \emph{connected $(k$-$)$partition}, if the subgraph induced by the vertices in $T_i$ is connected for each $1\leq i\leq k$.

The typical modeling objective in such connected partition problems is to balance sizes among the $k$ parts.
Sometimes one needs to consider a vertex-weighted generalization e.g.~weights representing the required amount of work at the entity corresponding to the vertex. 
The two classical balancing objectives for such $k$-partitions are to maximize the total weight of the lightest part, or to minimize the weight of the heaviest part. 
These objectives yield the following two versions of the \emph{balanced connected partition problem} (BCP).
\begin{problem}{Max-Min BCP (Min-Max BCP)}
	Input:&	A vertex-weighted graph $G=(V,E,w)$ where $w:V\rightarrow \mathbb N$, and $k \in \mathbb N$.\\
	Task:& Find a connected $k$-partition $T_1, \dots, T_k$ of $G$ maximizing  $\min_{i\in [k]}w(T_i)$ (minimizing $\max_{i\in[k]}w(T_i)$ resp.).
\end{problem}

On general graphs, both variants of BCP are $\textsf{NP}$-hard \cite{camerini1983complexity}, and hence the problems have been mostly studied from the viewpoint of approximation algorithms \cite{casel2020balanced,chataigner2007approximation,chen2020approximation,chen2019approximation,chlebikova1996approximating}.
Most of the known results are for small values of $k$, and there are some results also for special classes like grid graphs or graphs of bounded treewidth  (see related work section for further details).
The currently best known polynomial-time approximation for general graphs for any $k$ is a $3$-approximation for both \textsc{Max-Min} and \textsc{Min-Max BCP} by Casel et.~al.~\cite{casel2020balanced}.

Intuitively, an obstacle for getting a balanced connected partition is a large induced star, i.e.~a tree with one internal node and $c$ leaves, denoted by $K_{1,c}$.
We say a graph is \emph{$c$-claw-free} or \emph{$K_{1,c}$-free} if it does not contain an induced $K_{1,c}$ as subgraph. 
For such graphs,
we give a very efficient $(c-1)$-approximation algorithm for both the min-max and max-min objective.
In particular by setting $c=3$, we get a $2$-approximation on $K_{1,3}$-free graphs, better known as \emph{claw-free graphs}.

Claw free graphs have been widely studied by Seymour and  Chudnovsky in a series of seven papers under the name \emph{Claw-free graphs I-VII} (\cite{chudnovsky2007claw}-\cite{chudnovsky2012claw}\nocite{chudnovsky2008claw, chudnovsky2008claw3, chudnovsky2008claw4, chudnovsky2008claw5, chudnovsky2010claw}), who also provide a structure theorem for these graphs~\cite{chudnovsky2005structure}. 
Some interesting examples of such graphs are line graphs, proper circular interval graphs and de-Brujin graphs~\cite{compeau2011apply}.
Apart from their structural properties, claw-free graphs have been studied in the context of obtaining efficient algorithms for several interesting problems, see e.g.~\cite{hermelin2014parameterized,faenza2014solving,cygan2011dominating}.

Although, for $c>3$ our algorithm gives a worse guarantee than the algorithm by Casel et.~al.~\cite{casel2020balanced}, we note that their algorithm runs in $\bigO(\log(X^*) k^2|V||E|)$ time for {\maxminbcp} and in $\mathcal{O}\left(\log\left(X^* \right) |V|\,|E| \left(\log \log X^* \log\left(|V| w_{max} \right) + k^2 \right) \right)$ time for {\minmaxbcp}, where $X^*$ denotes the optimum value and $w_{\max} := \max_{v \in V} w(v)$ the maximum weight of a vertex,
whereas our algorithms give an $\bigO(\log(X^*)|E|)$ runtime for {\maxminbcp} and an $\bigO(|E|)$ runtime for {\minmaxbcp}.
Moreover, our algorithms are less technical and hence much easier to implement.
We prove the following statements.

\begin{restatable}{theorem}{MinMaxTHM}
	\label{thm::MinMax}
	Given a vertex-weighted $K_{1,c}$-free graph $G=(V,E,w)$ and $k \in \mathbb N$,
	a $(c-1)$-approximation for \textsc{Min-Max BCP} can be computed in $\mathcal O(|E|)$ time.
\end{restatable}

\begin{restatable}{theorem}{MaxMinTHM}
	\label{thm::MaxMin}
	Given a vertex-weighted $K_{1,c}$-free graph $G=(V,E,w)$ and $k \in \mathbb N$,
a $(c-1)$-approximation for \textsc{Max-Min BCP} can be computed in time $\O(\log(X^*)|E|)$,
where $X^*$ is the optimum value.
\end{restatable}

Since line graphs are $K_{1,3}$-free,
these results directly imply efficient approximations for the following edge-partition versions of BCP.
A $k$-partition of the edges of a graph, is called a \emph{connected edge $k$-partition}, if the subgraph induced by the edges in each part is connected.
In the problem  Min-Max (Max-Min) balanced connected edge partition (BCEP), one searches for a connected edge $k$-partition of an edge-weighted graph minimizing the maximum (resp. maximimzing the minimum) weight of the parts. 
This problem is equivalent to finding a connected $k$-partition of the vertices in the line graph of the input graph.
The best known approximation for BCEP is for graphs with no edge weight larger than $w(G)/2k$. For such graphs, \cite{chu2013linear} give an algorithm that
finds a connected edge $k$-partition, 
such that the weight of the heaviest subgraph is at most twice as large
as the weight of the lightest subgraph, implying
a $2$-approximation for \textsc{Min-Max} and \textsc{Max-Min BCEP}.
In comparison, our algorithms achieve the same approximation guarantee without restrictions on the edge weights.

\begin{corollary}
	\textsc{Min-Max BCEP} and \textsc{Max-Min BCEP}
	have $2$-approximations in polynomial time.
\end{corollary}

An extension of BCP is demanding for fixed (possibly heterogeneous) size targets for each of the $k$ parts. More precisely, given a graph $G$ and $w_1,\dots,w_k$ with $\sum_{i=1}^k w_i = n$, the task is to find a partition $T_1,\cdots, T_k$ where each $T_i$ has size $w_i$ and induces a connected subgraph. 
Such a connected $k$-partition with the fixed target weights exists for $G$ only if $G$ meets certain structural properties; a $K_{1,3}$ for example has no connected $2$-partition $T_1, T_2$ with $|T_1|=|T_2|=2$.
A characterization of when such a connected partition always exists was independently proved by Gy\H{o}ri~\cite{gyori1976division} and Lov\'{a}sz~\cite{lovasz1977homology}:
They showed that in any $k$-connected graph a connected $k$-partition satisfying the target weights always exists. This result is the famous {\gyori-\lovasz} Theorem (GL theorem, for short):
\begin{theorem}[{\gyori-\lovasz} Theorem \cite{gyori1976division,lovasz1977homology}]
	Given a $k$-connected graph $G=(V,E,w)$, $n_1,\dots,n_k\in \mathbb N$ such that $\sum_{i=1}^k n_i=|V|$, and $k$ terminal vertices $t_1,\cdots,t_k\in V$, there exists a connected $k$-partition $T_1,\cdots, T_k$ of $V$ such that for each $i\in [k]$, $|T_i|=w_i$ and $t_i\in T_i$.
\end{theorem}
Recently, the theorem was generalized to vertex-weighted graphs as:
\begin{theorem}[Weighted {\gyori-\lovasz} Theorem \cite{chandran2018spanning,chen2007almost,hoyer2019independent}]
	Given a vertex-weighted $k$-connected graph $G=(V,E,w)$, $w_1,\dots,w_k\in \mathbb N$ such that $\sum_{i=1}^k w_i= w(V)$, and $k$ terminal vertices $t_1,\cdots,t_k$, there exists a connected $k$-partition $T_1,\cdots, T_k$ of $V$ such that $w_i-\wmax<w(T_i)<w_i+\wmax$, and $t_i\in T_i$ for each $i\in [k]$, where $\wmax$ is the largest vertex weight.
\end{theorem}
We will refer to the partition guaranteed by the (weighted) GL theorem as \emph{GL partition}. We will however not consider the terminal vertices in the GL partitions in this work.

The GL theorem has found some applications in the field of algorithms. Chen et.~al.~\cite{chen2007almost} use it for proving the existence of low-congestion \emph{confluent flows} in $k$-connected graphs. Further, L\"{o}wenstein et.~al.~\cite{lowenstein2009spanning} and Chandran et.~al.~\cite{chandran2018spanning} use it for finding spanning trees with low \emph{spanning tree congestion}. 
Perhaps, the reason why such a strong combinatorial statement has not found further applications is that we do not know how to efficiently compute GL partitions. About five decades after the discovery of the GL theorem, polynomial time algorithms for finding a GL partition (even in the unweighted case without terminals) are only known for $k\le 4$ 
~\cite{suzuki1990linear,wada1993efficient,hoyer2019independent}.
The fastest algorithm for general $k$ takes $\Omega(2^n)$ time~\cite{chandran2018spanning,IssacThesis}.
Neither there are any impossibility results to exclude efficient computability of such partitions. 
Even when $k$ is part of the input, a polynomial time algorithm is not ruled out.

The absence of efficient algorithms for finding exact GL partitions motivates finding GL-style partitions that approximately satisfy the weight targets. 
In this paper we present polynomial time algorithms that does this.
First we give an algorithm for a ``half-bounded'' GL partition, in the sense that we can guarantee an approximate upper or lower bound on the weight of the parts.
\begin{restatable}{theorem}{oneSideGL}
	\label{thm:one-side}	
		Let $G=(V,E,w)$ be a $k$-connected vertex-weighted graph and $w_1,\dots,w_k\in \mathbb N$ with $\sum_{i=1}^k w_i = w(G)$, and $\min_{i\in[k]} w_i \geq \max_{v\in V}w(v)$.
		A connected $k$-partition $T_1, \dots, T_k$ of $V$  such that either $w(T_i) \geq \frac{1}{3} w_i$ for every $i \in [k]$ (lower-bound version) or  $w(T_i)\leq 3w_i$ for every $i\in [k]$ (upper-bound version) can be computed  in time $\O(k|V|^2|E|)$.
\end{restatable}
\noindent
We then extend this result to a lower and upper bounded partition.
\begin{restatable}{theorem}{bothSideGL}
	\label{thm:both-side}
	Let $G=(V,E,w)$ be a $k$-connected vertex-weighted graph and $w_1,\dots,w_k\in \mathbb N$ with $\sum_{i=1}^k w_i = w(G)$, and $\min_{i\in[k]} w_i \geq \max_{v\in V}w(v)$, and $r := \frac{\max_{i \in [k]} w_i}{\min_{j \in [k]} w_j}$. 
	Then, a connected $k$-partition $T_1, \dots, T_k$ of $V$ such that $\frac{1}{3}w_i \leq w(T_i) \leq \UpBoundGL w_i$ for every $i \in [k]$ can be found in time $\bigO(k|V|^2|E|)$.
\end{restatable}

In particular, \cref{thm:both-side} implies the following approximately balanced partition of $k$-connected graphs.
\begin{restatable}{corollary}{BalancedBothSideGL}
	\label{cor:balanced}
	Let $G=(V,E,w)$ be a $k$-connected vertex-weighted graph such that ${w(G)} \geq {k}\max_{v\in V}w(v)$.	Then, a connected $k$-partition $T_1, \dots, T_k$ of $V$ such that $\frac{1}{3}\left\lfloor\frac{w(G)}{k}\right\rfloor\leq w(T_i) \leq 
	3\left\lceil\frac{w(G)}{k}\right\rceil$ for every $i \in [k]$ can be  found in time $\bigO(k|V|^2|E|)$.	
\end{restatable}

To the best of our knowledge, these are the first polynomial time algorithms that \emph{approximate} the GL theorem.
We believe that such an efficient approximation will result in the theorem being used for developing algorithms in the future.
Especially, we are hopeful that the both-side approximation for balanced connected partition of $k$-connected graphs will find applications.
We remark, however that for the above mentioned applications of confluent flows and spanning tree congestion, the terminal vertices are essential and hence our algorithms cannot be used.
An interesting future direction would be to extend our results to the setting with terminals.

Observe that Corollary~\ref{cor:balanced} in some sense yields a 3-approximation simultaneously for \textsc{Min-Max} and \textsc{Max-Min BCP} in $k$-connected graphs. In this regard, it is interesting to note that the $+/-\wmax$ slack given in the weighted GL theorem is enough to retain hardness in the following sense: even for $k=2$, \textsc{Min-Max BCP}
and \textsc{Max-Min BCP}
remain strongly $\textsf{NP}$-hard  when restricted to $2$-connected graphs; and the corresponding hardness-proof given in~\cite{chataigner2007approximation} also constructs an instance with ${w(G)} \geq {k}\max_{v\in V}w(v)$. 
This hardness can be extended to $k$-connected graphs for any fixed $k\geq 2$ (see~\cite{chataigner2007approximation}[Theorem 3] for more details).

\subsection{Related work}
Both variants of \textsc{BCP} were first introduced for trees~\cite{perl1981max,kundu1977linear}.
Under this restriction, a linear time algorithm was provided for both variants in~\cite{frederickson1991optimal}.
This is particularly important since different heuristics transform the original instance to a tree to efficiently 
solve the problem, see~\cite{chu2013linear,zhou2019balanced}.
For both variants of \textsc{BCP}, a $3$-approximation is given in~\cite{casel2020balanced}, which is the best known approximation in polynomial time.
With respect to lower bounds, it is known that there exists no approximation  for $\textsc{Max-Min BCP}$ with a ratio below~$6/5$, unless $\textsf{P}\neq \textsf{NP}$~\cite{chataigner2007approximation}.
For the unweighted case, a $\frac{k}{2}$-approximation for \textsc{Min-Max BCP} with $k \geq 3$, is  given in~\cite{chen2019approximation}.%

Balanced connected partitions for fixed small values of $k$, denoted $\textsc{BCP}_k$, have also been studied extensively. The restriction $\textsc{BCP}_2$, i.e.~balanced connected bipartition, is already \textsf{NP}-hard~\cite{camerini1983complexity}. On the positive side, a 
$\frac{4}{3}$-approximation for $\textsc{Max-Min BCP}_2$ is given in~\cite{chlebikova1996approximating}, and in~\cite{chen2019approximation} this result is used to derive a $\frac{5}{4}$-approximation for $\textsc{Min-Max BCP}_2$. Considering
 tripartitions, $\textsc{Max-Min BCP}_3$ and $\textsc{Min-Max BCP}_3$ can be approximated  with ratios $\frac{5}{3}$ and $\frac{3}{2}$, respectively~\cite{chen2020approximation}.

Regarding special graph classes, BCP has been investigated in grid graphs and series-parallel graphs.
While it was shown that BCP is $\textsf{NP}$-hard for arbitrary grid 
graphs~\cite{becker1998max},
the \textsc{Max-Min BCP} can be solved in polynomial
time for ladders, i.e., grid graphs with two rows~\cite{becker2001polynomial}.
For the class of series-parallel graphs, Ito~et.~al.~\cite{ito2006partitioning} 
observed that BCP remains weakly $\textsf{NP}$-hard (by a simple reduction from the Partition problem) and gave a pseudo-polynomial-time algorithm 
for both variants of BCP. 
They
also showed that their algorithm can be extended to graphs
with bounded tree-width.

The GL Theorem was independently proved by  Gy\H{o}ri~\cite{gyori1976division} and Lov{\'{a}}sz~\cite{lovasz1977homology}. 
{\gyori} used an elementary graph theoretic approach while {\lovasz} used ideas from topology.
{\lovasz}'s proof also works for directed graphs.
The Gy\H{o}ri-Lov{\'{a}}sz Theorem is extended to weighted directed graphs by Chen et.~al.~\cite{chen2007almost} and Gy\H{o}ri's original proof was generalized to weighted undirected graphs by Chandran et.~al.~\cite{chandran2018spanning}.
Both papers only gave upper bounds of $w_i+\wmax$ on the weight of partition $T_i$ and did not provide any lower bounds.
Later Hoyer~\cite{hoyer2019independent} showed that the method of Chandran et.~al.~\cite{chandran2018spanning} can be also extended to give the lower bound $w_i-\wmax$, even for directed graphs.
Polynomial algorithms to also compute GL partitions are only known for the particular cases $k=2,3,4$~\cite{suzuki1990linear,wada1993efficient,hoyer2019independent} and all $k\ge 5$ are still open.

%% file: prelims.tex
%

By $\mathbb N$ we denote the natural numbers without zero. We use $\left[k\right]$ to denote the set $\left\{1,\ldots,k\right\}$.

All the graphs that we refer to in this paper are simple, finite and connected.
Consider a graph $G=(V,E)$.
We denote by $V(G)$ and $E(G)$ the set of vertices and edges of $G$ respectively, and if the graph we refer to is clear, we may simply write $V$ and $E$.
For a set of vertex sets $\calS \subseteq 2^V$ we use $V(\calS)$ to denote $\bigcup_{S \in \calS} S$. 
We denote an edge $e=\left\{u,v\right\}\in E(G)$ by $uv$ and the neighborhood of a vertex $v\in V$ in $G$ by $N_G\left(v\right)=\left\{u\in V\mid uv\in E(G)\right\}$. 
Similarly we denote the neighborhood of a vertex set $V'\subseteq V$ in $G$ by $N_G\left(V'\right)$, that is $\bigcup_{v \in V'} N_G\left(v\right) \setminus V'$. We may omit the subscript $G$ when the graph is clear from the context.
We use $\Delta(G)$ to denote the maximum degree of $G$.

We denote a \emph{vertex-weighted graph} by $G=\left(V,E,w\right)$ where $w$ is a function assigning integer weights to vertices $w\colon V\rightarrow \mathbb N$, and $V$ and $E$ are vertex and edge sets. We denote by $w_{\text{min}}$ and by $w_{\text{max}}$, $\text{min}_{v\in V}w\left(v\right)$ and $\text{max}_{v\in V} w\left(v\right)$, respectively. For any $V'\subseteq V$, we use $w(V')$ to denote the sum of weights of the vertices in $V'$.
For a subgraph $H$ of $G$ we use $w\left(H\right)$ to denote $w(V(H))$, and refer to it as the \emph{weight of the subgraph $H$}.
For a rooted tree $T$ and a vertex $x$ in $T$, we use $T_x$ to denote the rooted subtree of $T$ rooted at $x$.

For $V'\subseteq V$ we denote by $G[V']$ the graph induced by $V'$, i.e.~$G[V']=(V',E')$ with $E'=E\cap (V'\times V')$. For vertex-weighted graphs, induced subgraphs inherit the vertex-weights given by $w$. For $V'\subseteq V$ we also use $G-V'$ to denote the subgraph  $G[V\setminus V']$. Similarly, if $V'$ is a singleton $\left\{v\right\}$ we also write $G-v$.
For graphs $G_1$ and $G_2$, we use $G_1\cup G_2$ to denote the graph on vertices $V(G_1)\cup V(G_2)$ with edge set  $E(G_1)\cup E(G_2)$.

Let $\calU=\left\{ U_1, \dots, U_r \right\}$ be such that each $U_i\subseteq V$.  We call $\calU$  a \emph{connected packing} of $V$ if each $G[U_i]$ is connected, and the sets  in $\calU$ are pairwise disjoint. 
A connected packing $\calU$ is called \emph{connected vertex partition} (\cvp) of $V$, if also $\cup_{i=1}^rU_i=V(G)$.
We denote a {\cvp} that has $k$ vertex sets as {\cvp}$_k$.
For any $\calU'\subseteq \calU$, we define $V\left(\mathcal{U}'\right) := \bigcup_{U' \in \mathcal{U}'} U'$, and the weight $w\left(\calU'\right):=w\left(V\left(\calU'\right)\right)$.
Let $\mathcal{I}$ be an interval.
If $\calU$ is a {\cvp} and $w\left(U_i\right) \in \mathcal{I}$ for all $i \in [r]$, then we say that $\mathcal{U}$ is an \emph{$\mathcal{I}$-connected vertex $(r$-$)$partition} ($\mathcal{I}$-$\cvp_k$ or just $\mathcal{I}$-\cvp) of $V$.
If $\calU$ is a {connected-packing} and $w\left(U_i\right) \in \mathcal{I}$ for all $i \in [r]$, then we say that $\mathcal{U}$ is a \emph{$\mathcal{I}$-connected packing} of $V$.

%% file: SubAlgorithmEdgePartition.tex
\noindent
\cref{thm::MinMax} now follows from running {\algBalancedPartition} with $\lambda = \max\{\wmax,\frac{w(G)}{k}\}$. Note that this choice of $\lambda$ is a trivial lower bound for the optimum value.

\medskip
As already mentioned, to prove \cref{thm::MaxMin},  we first need to find an input parameter $\lambda$ for Algorithm {\algBalancedPartition} that provides the desired $(c-1)$-approximation.

Let $(G,k)$ be an instance of {\maxminbcp}, where $G$ is a $K_{1,c}$-free graph.
Let $X^*$ be the optimal value for the instance $(G,k)$.
For any given  $X \leq w(G)/k$, we design an algorithm that either gives a $[\lfloor X/(c-1) \rfloor, \infty)$-{$\cvp_{k}$}, or reports that $X>X^*$.
Note that $X^*\le w(G)/k$.
Once we have this procedure in hand, a binary search for the largest $X$ in the interval $\left(0, \left\lceil w(G)/k \right\rceil \right]$ for which we find a $[\lfloor X/(c-1) \rfloor, \infty)$-{$\cvp_{k}$} can be used to obtain an approximate solution for {\maxminbcp}.\\

\noindent
\textbf{Algorithm \algmaxmin:}
First remove all vertices of weight more than $\lambda=\lfloor X/(c-1)\rfloor$ and save them in $H$.
Then save the connected components of weight less than $\lambda$ in $\calQ$.
Let $\calV = \{V_1,\dots,V_{\ell}\}$ be the connected components of $G - (H \cup V(\calQ))$.
Apply algorithm {\algBalancedPartition} on each $G[V_i]$ with $\lambda$ as input parameter to obtain $\calS^i = \{S^i_1, \dots, S^i_{m_i}\}$ for every $i \in [\ell]$.
If for some $i \in [\ell]$ the weight $w(S_{m_i})$ is less than $\lambda$, then merge this vertex set with $S_{m_i - 1}$ and accordingly update $\calS^i$.
Further, compute a $(\lambda, \infty)$-{$\cvp_{|H|}$} $\calS^H$ of $G[H \cup V(\calQ)]$ as follows: for each $h\in H$, we will have a set $S^h\in \calS^H$ with $h\in S^h$; we add each $Q \in \calQ$ to some $S^h$ such that $h \in N(Q)$. 
Let $\calS = \calS^H \cup \bigcup_{i=1}^{\ell} \calS^i$.
If $|\calS| \geq k$, then merge connected sets arbitrarily in $\calS$ until $|\calS| = k$ and return $\calS$.
If $|\calS|<k$, report that $X>X^*$.
\newline

We point out that a $[\lambda, \infty)$-{$\cvp_{j}$} with $j > k$, can easily be transformed to a $[\lambda, \infty)$-{$\cvp_{k}$}, since the input graph is connected. It is not hard to see that if algorithm {\algmaxmin} returns $\calS$ then this is a $[\lambda, \infty)$-{$\cvp_{k}$} of $V$.
The most complicated part of proving that   {\algmaxmin} works  correctly is showing that if it  terminates with $|\calS| < k$ and reports $X>X^*$ that this is indeed true.

\begin{lemma}
	\label{lemma::Max-MinKeyLemma}
	If Algorithm {\algmaxmin} terminates with $|\calS| < k$, then $X>X^*$.
\end{lemma}
\begin{proof}
	Let $H,\calQ,\calV = \{V_1,\dots,V_{\ell}\}$ be the computed vertices and connected vertex sets in the algorithm for $\lambda=\lfloor X/(c-1)\rfloor$, respectively.
	Recall that $w(V_i) \geq \lambda$ for every $V_i \in \calV$ and $w(Q) < \lambda$ for every $Q \in \calQ$.
	Let $\calS^* = \left\{S^*_1, \dots, S^*_k\right\}$ be an optimal solution of $(G,k)$, i.e., $\calS^*$ is an $[X^*, \infty)$-{$\cvp_{k}$} of $V$. 	
	Consider the sets $\calV^{H \cup \calQ} := \left\{S^*_i \in \calS^*|  S^*_i \cap (H \cup V(\calQ)) \ne \varnothing \right\}$, $\calV^{1} := \left\{S^*_i \in \calS^*|  S^*_i \cap V_1 \neq \varnothing \right\} \setminus \calV^{H \cup \calQ}$, $\dots$, $\calV^{\ell} := \left\{S^*_i \in \calS^*|  S^*_i \cap V_{\ell} \neq \varnothing \right\} \setminus \mathcal{V}^{H \cup \calQ}$.
	We claim that these sets are a partition of $\calS^*$. This follows directly from the fact that $H$ separates all $V_i \in \calV$ and all $Q \in \calQ$ from each other.
	That is, for an $i \in [k]$ and $j \in [\ell]$ the  connected vertex set $S^*_i$ with $S^*_i \cap V_j \neq \varnothing$ and $S^*_i \cap V \setminus V_j \neq \varnothing$ contains at least one $h \in H$ and hence $S^*_i \in \calV^{H \cup \calQ}$.	Otherwise, if $S^*_i \subseteq V_j$, then $S^*_i \in \calV^j$.
	
	Suppose {\algmaxmin} terminates with $|\calS|<k$ although $X\le X^*$. 	
	We show that $\left|\mathcal{V}^{H \cup \calQ}\right| \leq |H|$ and $\left| \mathcal{V}^{i}  \right| \leq |\calS^{i}|$ for every $i \in [\ell]$, implying that $\left|\calS^*\right|\leq |H| + \sum_{i=1}^{\ell} |\calS^{i}| = |\calS|<k$, which contradicts $\left|\calS^*\right|=k$.
	
	First, we show $\left|\calS^{H \cup \mathcal{Q}}\right| \leq |H|$.
	For this, it is sufficient to prove that $S^*_i \cap H \neq \varnothing$ for each $S^*_i \in \calV^{H \cup \calQ}$ as $\calS^*$ is a partition of $V$.
	We prove this by contradiction.
	Suppose there is an $S^*_i\in \calV^{H \cup \calQ}$, such that $S^*_i \cap H = \varnothing$.
	This implies that $S^*_i \subseteq Q$ for some $Q \in \calQ$, since $H$ separates every $Q \in \calQ$ from every other $Q' \in \calQ \setminus \{Q\}$ and from the vertices $V \setminus (H \cup V(\calQ))$.
	Thus, $w(S_i^*) \leq w(Q) < \lambda$ by the definition of $\calQ$ and therefore $w(S_i^*)<\lambda= \lfloor X/(c-1)\rfloor \le \lfloor X^*/(c-1)\rfloor$, contradicting $\min_{i \in [k]} w(S^*_i)=X^*$.
	
	It remains to show that $\left| \mathcal{V}^{i}  \right| \leq |\calS^{i}|$ for every $i \in [\ell]$.
	Fix an $i \in [\ell]$ and let $G[V_i]$ with $\lambda$ be the input when calling algorithm {\algBalancedPartition}.
	Observe that the input is valid, since $G[V_i]$ is connected by definition and $w(G[V_i])\geq \lambda \geq \max_{v \in V_i} w(v) $ as $H$ contains all vertices that have weight more than $\lambda$.
	Algorithm {\algBalancedPartition} provides a {\cvp} $\calS^i = \{S^i_1,\dots,S^i_{m_i}\}$ of $V_i$ with $w(S^i_j) \in [\lambda, (c-1)\lambda)$ for every $j \in [m_i - 1]$ and $w(S^i_{m_i}) < (c-1)\lambda$. 
	Consider $\calS^i$ before merging, i.e.~we do not merge $S^i_{m_i}$ to $S^i_{m_i-1}$ in the algorithm {\algmaxmin} if $w(S_{m_i}) < \lambda$.
	That is, $w(S_m) < \lambda$ is possible, and we need to show $\left| \mathcal{V}^{i}  \right| \leq m_i -1 = |\calS^{i}| -1$.
	Observe for $S^* \in \calV^{i}$ that $S^* \subseteq V_i$, i.e.~$\sum^{m_i}_{j=1} w(S^i_j) \geq \sum_{S^* \in \calV^{i}} w(S^*)$.
	As a result, we have $|\calS^{i}| X \geq |\calS^{i}| (c-1)\lambda > \sum^{m_i}_{j=1} w(S^i_j) \geq \sum_{S^* \in \calV^{i}} w(S^*) \geq |\calV^{i}|X^*$.
	Consequently, by $X \leq X^*$ we obtain  $|\calV^{i}| < |\calS^{i}|$, which leads  to $|\calV^{i}| \leq |\calS^{i}| - 1$.
\end{proof}

%% file: GyoeriLovasz.tex

Our algorithms for the approximate GL theorems are based mainly on the following combinatorial lemma concerning certain vertex separators, that leads to useful structures in $k$-connected graphs.
Let $G = (V,E,w)$ be a connected vertex-weighted graph and let $\lambda$ be an integer.
We say $s \in V$ is a $\lambda$-separator if all connected components of $G - \{s\}$ weigh less than $\lambda$.
We say $G$ is \emph{$\lambda$-dividable} if there is a $[\lambda,\infty)$-{\cvp}$_2$  of $V$.

\begin{lemma}[\cite{casel2020balanced}]
	\label{lemma::divideInto2Comp}
	Let $G = (V,E,w)$ be a connected vertex-weighted graph and let $\lambda > w_{\max}$ be an integer.
	If $w(G) > 3(\lambda - 1)$, then either $G$ is $\lambda$-dividable or there is a $\lambda$-separator.
	Furthermore, finding the connected vertex sets in case $G$ is $\lambda$-dividable and finding the $\lambda$-separator in the other case can be done in $\mathcal{O}(|V|\,|E|)$ time.
\end{lemma}

\subsection{Bounded Partition for $k$-connected Graphs}
\label{section::GL}
\input{LowerBoundGL}

\subsection{Both-side Bounded Partition for $k$-connected Graphs}
\label{sec:double-side}
\input{DoubleBoundedGL}

%% file: LowerBoundGL.tex
In this section, we give an algorithm for computing approximate GL partitions with one-side approximation bound (either lower bound or upper bound), thus proving \cref{thm:one-side}. 
For this, we first prove the following theorem, from which \cref{thm:one-side} follows as below.
\begin{theorem}
	\label{thm::GLpartition}
	Let $G=(V,E,w)$ be a $k$-connected vertex-weighted graph and let $w_1,\dots,w_k\in \mathbb N$ with $\sum_{i=1}^k w_i = w(G)$, and $\min_{i\in[k]} w_i \geq \max_{v\in V}w(v)$.
 A set of connected vertex sets $\T = \{T_1, \dots, T_\ell\}$ with $\ell \leq k$ and $\alpha w_i \leq w(T_i) \leq 3 \alpha w_i$ for every $i \in [\ell]$ can be computed in time $\O(k|V|^2|E|)$.
	Moreover, if $\ell < k$, then $\T$ is also a {\cvp} of $V$.
\end{theorem}

By \cref{thm::GLpartition} we can derive \cref{thm:one-side} using $\alpha=1/3$ and $\alpha=1$  for the lower bound and upper bounded version, respectively.  \cref{app:BoundedGL} gives a detailed proof.

In the following we always assume that $w_1, \dots, w_k$ is sorted in descending order.
To now give the algorithm proving \cref{thm::GLpartition}, we make use of \cref{lemma::divideInto2Comp}. For this, we first need to ensure that $w_{\max} < \alpha w_k$.
Therefore, we perform a preprocessing step until we reach an instance that satisfies $w_{\max} < \alpha w_k$.
We give this preprocessing step in \cref{app:preprocess}.
After this step, we can assume that we have a $k$-connected graph $G = (V,E,w)$ and natural numbers $w_1, \dots, w_k$ sorted in descending order, where $\sum_{i=1}^k w_i \leq w(G)$ and $\wmax < \alpha w_k$. 

On such a graph $G$ we then gradually build a packing $\T$ with the help of \cref{lemma::divideInto2Comp}. During our algorithm to build  $\T$ we ensure that at each step $\T=\left\{ T_1,T_2,\dots,T_{i-1} \right\}$ where each $T_j\in \T$ is a connected vertex set with weight in $[\alpha w_j , 3\alpha w_j]$ for each $j\in [i-1]$.
We then search in the remaining graph for the next set $T_i$ and always use $\og$ to denote the graph $G\setminus V(\T)$.	We say a connected subgraph is $i$-\emph{small} if it has weight less than $\alpha w_i$ and $i$-\emph{big} otherwise. In case we reach a situation, where  $\og$ has no connected component that is $i$-\emph{big}, we have to alter the already built sets $T_1,T_2,\dots,T_{i-1}$ to build $T_i$. For this, we use $\T_a$ to denote  the set of all $T_j \in \T$ that have \emph{no} $(\alpha w_j)$-separator, and $\T_b$ to denote the set of all $T_j \in \T$ that have an $(\alpha w_j)$-separator. For $T_j\in \T_b$ with an $(\alpha w_j)$-separator $s$  we use $C(T_j)$ to denote the connected components of $G[T_j \setminus \left\{ s \right\}]$ (if there is more than one $(\alpha w_j)$-separator, fix one of them arbitrarily). The following Algorithm \algBoundedGL{} formally explains our routine to build $\T$.



\paragraph*{Algorithm \algBoundedGL}
\begin{enumerate}
\item
	\label{algBoundedGL::Init}
	Initialize $\T := \varnothing$ as container for the desired  connected-vertex-packing $T_1, \dots, T_k$ of $G$
	and initialize $i:=1$ as an increment-variable.
	
\item
	While $\og=G\setminus V(\T)$ is not the empty graph: \texttt{//main loop}
	\begin{enumerate}[label*=\arabic*.,ref=\theenumi.\arabic*]

	\item
	\label{algBoundedGL::For_i}
	Find a connected vertex set $T_{i}$ having weight in $[\alpha w_{i},3 \alpha w_{i}]$,
	add $T_{i}$ to $\T$, and increment $i$ by one.
	If $i=k+1$ then terminate the algorithm. \\
	\texttt{// See \cref{lem:bigcomp} for correctness of this step}
		\item
	\label{algBoundedGL::RemoveQ}
	While $\og$ is not empty and has no $i$-big connected component:
	\texttt{//inner loop}\\
	Pick an $i$-small connected component $Q$ of $\og$. 
	Pick a $T_j\in \T$ such that
	either $T_j\in \T_a$ and $Q$ has an edge to $T_j$ (Case 1), or $T_j\in\T_b$ and $Q$ has an edge to some component $Q'\in C(T_j)$ (Case 2).
	\texttt{// The occurrence of at least one of these cases is shown in \cref{lemma::connectionQ}.}\\
	If $w(T_j\cup Q)\le 3\alpha w_j$ then update $T_j$ to $T_j \cup Q$. 
	Otherwise:
	\begin{enumerate}[label*=\arabic*.,ref=\theenumii.\arabic*]
		\item 
		\label{algBoundedGL::Case1}
		Case 1 ($T_j\in \T_a$): 
		Apply the following 
		\textbf{\emph{Divide-routine}} on $T_j\cup Q$: 
		Use \cref{lemma::divideInto2Comp} to compute a $[\alpha w_j, \infty )$-CVP$_2$ $V_1, V_2$ of $T_j\cup Q$.
		Set $T_{j} = V_2$ (i.e. $V_1$ goes to $\og$).
		\item
		\label{algBoundedGL::Case2}
		Case 2 ($T_j\in\T_b$):
		remove $Q'$ from $T_j$ (i.e.~$Q'$ goes back to $\og$) if $T_j \cup Q$ is not $\alpha w_j$-dividable. Otherwise, apply divide routine on $T_j \cup Q$.
	\end{enumerate}
	\end{enumerate}

\end{enumerate}
To prove the correctness of the algorithm, we show that the following invariant is maintained.
\begin{lemma}
	\label{lem:invT}
Algorithm \algBoundedGL{} maintains a packing $\T=\left\{ T_1,T_2,\dots,T_{i-1} \right\}$ where each $T_j\in \T$ is a connected vertex set having weight in $[\alpha w_j , 3\alpha w_j]$.
\end{lemma}
\begin{proof}
	We increment $i$ only in Step~\ref{algBoundedGL::For_i}.
	Before incrementing $i$, we add $T_i$ to $\T$ while ensuring that $w(T_i)\in [\alpha w_i, 3\alpha w_i]$ and $G[T_i]$ is connected.
In \cref{lem:bigcomp}, we prove that whenever the divide routine is about to be executed, there is an $i$-big component in $\og$, ensuring the existence of such a $T_i$.
	Once a $T_j$ is added to $\T$, it is then modified only in Step~\ref{algBoundedGL::RemoveQ}.
	So let us look into how it gets modified in Step~\ref{algBoundedGL::RemoveQ}.
	If the condition $w(T_j\cup Q)\le 3\alpha w_i$ is satisfied then it is clear that the new $T_j=T_j\cup Q$ also satisfies the weight constraints.
	Since $Q$ has an edge to $T_j$ and $T_j$ and $Q$ each were connected, it is also clear that the new $T_j$ remains connected.
	So now consider the case when $w(T_j\cup Q)>3\alpha w_i$.
	In Case 1 ($T_j\in\T_a$),
	we call the divide routine and the new $T_j$ is the set $V_2$ returned by the routine.
	The set $V_2$ is connected due to the property of the divide routine.
	To see that it also satisfies the weight constraints, observe that $w(V_1\cup V_2)$ is at most $4\alpha w_j$ as $w(T_j)$ was at most $3\alpha w_j$ and $w(Q)<\alpha w_i\leq \alpha w_j$. Since $w(V_1),w(V_2)\ge \alpha w_j$, we then have $w(V_2)\in [\alpha w_j, 3\alpha w_j]$.
	So it only remains to consider Case 2 ($T_j\in \T_b$).
	The case when $T_j \cup Q$ is $\alpha w_j$-dividable is analog to Case 1. 
	We know $w(Q')<\alpha w_j$ by definition.
	Also, since $w(T_j\cup Q)$ was more than $3\alpha w_j$ and $w(Q)<\alpha w_j$, we have that $w(T_j)$ was at least $2\alpha w_j$.
	Thus the new $T_j=T_j\setminus Q'$ has weight in $[\alpha w_j,3\alpha w_j]$.
	Also, the new $T_j$ is connected by the definition of $Q'$.
\end{proof}

It is clear from the algorithm that termination occurs only if $\og$ is empty or $i=k+1$.
Then using Lemma~\ref{lem:invT}, it is clear that $\T$ contains the required packing as claimed in Theorem~\ref{thm::GLpartition}, provided that Step~\ref{algBoundedGL::RemoveQ} runs correctly and terminates, which we prove below.

\begin{lemma}
	\label{lem:inner_loop}
	The inner loop runs correctly and terminates after at most $|V|^2$ iterations.
\end{lemma}
The idea for the  proof of \cref{lem:inner_loop} is that if the divide routine is executed in the inner loop then an $i$-big component is created, terminating the inner loop, and in the other case either a connected component is deleted from $\og$ or new vertices are added to a connected component in $\og$.
A full proof  and the runtime analysis are given in \cref{app:BoundedGL} and~\ref{app:runtimeBGL}, respectively.

It is tempting to think that one could use \cref{thm::GLpartition} to derive  a  {\cvp} $\T=\{T_1,\dots,T_k\}$ such that $\alpha w_i \leq w(T_i) \leq 3 \alpha w_i$ for each $i\in[k]$. 
If Algorithm \algBoundedGL{} terminates with $\ell<k$, then $\T$ is a partition of the vertices in $G$, and we only have trouble with the lower bound on $T_j$ for $\ell<j\leq k$. Otherwise, if it terminates with $\ell=k$, then  $\T$ satisfies all lower bounds, but might not be a partition. Assigning the remaining vertices in $G$ to turn $\T$ into a {\cvp} in this case might yield violations of the upper bound. Since $\alpha=1$ yields the first, and $\alpha=\frac 13$ the second case, one might think that choosing the correct $\alpha$ in between would result in a {\cvp} with $\ell=k$. Unfortunately, Algorithm \algBoundedGL{} does not have a monotone behaviour w.r.t.~$\alpha \in (\frac 13, 1)$ in the sense that for two values $\frac 13<\alpha_1<\alpha_2<1$, the case $\ell=k$ for $\alpha_1$ does not imply $\ell=k$ for $\alpha_2$. Thus, even if we could prove the existence of an optimal value for $\alpha$, we have no way to search for it.

%% file: DoubleBoundedGL.tex

In this section, we prove \cref{thm:both-side} by giving a both-side bounded approximate GL partition.
The full correctness and runtime proofs of the algorithm can be found in \cref{app:double-side}.

For achieving a simultaneous lower and upper bounded partition, as a starting point, we apply \cref{thm::GLpartition} with $\alpha = \frac{1}{3}$ obtaining a lower and upper bounded packing $\T = \{T_1, \dots, T_k\}$ with
$\frac{1}{3} w_i \leq w(T_i) \leq w_i$ for every $i \in [k]$.
As long as $\T$ is not a {\cvp} of $V$ yet, we transfer a subset of the remaining vertices $V \setminus V(\T)$ through a path in an auxiliary graph to elements in $\T$,
while making sure that for each $i$, $\frac{1}{3}w_i\le w(T_i)\le \UpBoundGL w(T_i)$.
We call one such transfer a \emph{transferring-iteration}.
We define $\T^* := \{T_1, T_2,\dots,T_j\}$ where $j$ is the smallest number such that $w(T_i) \geq w_i$ for $i \in [j]$ and $w(T_{j+1}) < w_{j+1}$. In case of $w_a = w_b$ and $w(T_a) \geq w_a$, but $w(T_b) < w_b$ we  assume that $a < b$.
Note that this is easily realizable by a relabeling of indices.
Observe that $\T^* = \varnothing$ if $w(T_1) < w_1$.
As a measure of progress, we guarantee in each transferring-iteration that either 
the cardinality of $\T^*$  increases, or 
the number of vertices in $V(\T)$ increases. 
Also, the cardinality of $\T^*$ is non-decreasing throughout the algorithm.
Note that if $T_i\in \T^*$ for all $i$, then it follows that $w(T_i)=w_i$ for all $i$ and moreover, $\T$ is a {\cvp}.

Let $\T = \{T_1, \dots, T_k\}$ be a connected packing of $V$ in $G$ with $w(T_i) \leq \UpBoundGL w_i$ for every $i \in [k]$.
We use $\calQ$ to denote the vertex sets forming the connected components of $G[V \setminus V(\T)]$. 
We define $\T^+ := \{T_i \in \T \mid w(T_i) \geq w_i\}$ and $\T^- := \T \setminus \T^+$.
Note that $\T^* \subseteq \T^+$. 
Analogous to section~\ref{section::GL}, we define $\T^+_a $ as the set of $T_i\in \T^+$ that do \emph{not} have a $w_i$-separator vertex and $\T^+_b$ to be the ones in $\T^+$ having a $w_i$-separator.
For $T_i\in \T^+_b$, we use $s(T_i)$ to denote its $w_i$-separator (if there are multiple we fix one arbitrarily) and $C(T_i)$ to denote the vertex sets forming the connected components of $G[T_i \setminus\{s(T_i)\}]$. 
We say a vertex $v \in V$ or a vertex set $V' \subseteq V$ is \emph{$\T$-assigned} if $v \in V(\T)$ or $V' \subseteq V(\T)$, respectively.
That is, the set of $\T$-assigned vertices is $V(\T)$ and $V(\calQ)$ is the set of \emph{not $\T$-assigned} vertices.
We say \emph{$\T$ is pack-satisfied} if $|\T|=k$, each $T_j \in \T$ is connected, $w(T_j) \in [\frac{1}{3} w_j, \UpBoundGL w_j]$, and the vertex sets in $\T$ are pairwise disjoint.

We define the \emph{transfer-graph} $H = (\calV_H,E_H)$ as $\calV_H := (\bigcup_{T \in \T^+_b} C(T)) \cup \T^+_a \cup \T^- \cup \calQ$ and $E_H := \{ (V_1,V_2) \in \binom{\calV_H}{2} \mid N_G(V_1) \cap V_2 \neq \varnothing\}$.

\paragraph*{Algorithm \algDoubleBoundedGL}
\begin{enumerate}[label*=\arabic*.,ref=\arabic*]
	\item 
		\label{doubleGL:lowerBound}
		Apply \cref{thm::GLpartition} with $\alpha=\frac{1}{3}$ on $G$ to obtain a connected packing $\T = \{T_1, \dots, T_k\}$ with $w(T_i)\ge \frac{1}{3} w_i $ for every $i \in [k]$.
		\item
		\label{doubleGL:while}
	While $\calQ \neq \varnothing$:
	\begin{enumerate}[label*=\arabic*.,ref=\theenumi.\arabic*]
	\item
		\label{doubleGL:findPath}
	Find a minimal path in $H$ from $\calQ$ to $\T^-$.
	Let this path be $P_Q^{T_i}$ where $Q \in \calQ$ and $T_i \in \T^-$. \\
	\texttt{// Note that all vertices in $P_Q^{T_i}$ except the start and end vertex are in $\T^+_a \ \cup \ \bigcup_{T \in \T^+_b} C(T)$ by minimality of the path.~The existence of a path\\ from $\calQ$ to $\T^-$ is shown in \cref{lemma::setsOfP}.}
\item
	\label{doubleGL::transStep}
	Execute the {\algTransfer} routine given below,
	which augments vertices through the path $P_{Q}^{T_i}$ such that 
	$\T$ stays pack-satisfied, 
	and either 
	$|\T^*|$ increases, or 
	$|\T^*|$ remains the same and
	the number of $\T$-assigned vertices increases.
\end{enumerate}
\end{enumerate}

We need some more notations for describing the {\algTransfer} routine.
For $\calV_H' \subseteq \calV_H$ and $\T' \subseteq \T$ we define $\T'(\calV_H')$ as the set $\{T_i \in \T' \mid V(T_i) \cap V(\calV_H') \neq \varnothing\}$.
For $H' \subseteq H$ we define $ \T'(H') := \T'(\calV_H(H'))$, and $V(H') := V(\calV_H(H'))$.
With $|P_Q^{T_i}|$ we denote the length of the path $P_Q^{T_i}$, i.e.~the number of edges in $P_Q^{T_i}$. 
We define $P_Q^\ell$ as the vertex with distance $\ell$ to $Q$ in $P_Q^{T_i}$, where $P_Q^0 = Q$, and define $T(P_Q^\ell)$ for $\ell \in [|P_Q^{T_i}|]$ as the function which returns $T_j$ with $P_Q^\ell \subseteq T_j$.
For $\T' \subseteq \T$ we define $I(\T'):=\left\{ i\mid T_i\in \T' \right\}$. 

The {\algTransfer} routine transfers vertices through the path $P_Q^{T_i}$.
Our input is a pack-satisfied $\T$ and a $P_Q^{T_i}$ path according to Step~\ref{doubleGL::transStep} in algorithm {\algDoubleBoundedGL}.
By the minimality of the path $T_Q^{T_i}$, it is clear that $V_H(P_Q^{T_i}) \setminus \{Q,T_i\} \subseteq \T^+_a\cup \bigcup_{T \in \T^+_b} C(T)$.
That is, except for the destination $T_i$ we run only through vertex sets from $\T^+$ in $\T(P_Q^{T_i})$.
Roughly, our goal is to transfer vertices of $V(P_Q^{T_i} - T_i)$ to $T_i$, thereby changing the division of the vertex sets $\T(P_Q^{T_i})$ and preserving the vertex sets in $\T^*$.

We often need to do a \emph{truncate} operation on sets $T_j$ with
$w(T_j) > \UpBoundGL w_j$. We mean by \emph{truncate $T_j$} that we remove vertices from $T_j$ until $w_j \leq w(T_j) \leq \UpBoundGL w_j$ such that $T_j$ remains connected.
This can be done by removing a non-seperator vertex from $T_j$ until the weight drops below $\UpBoundGL w(T_j)$.
Note that any connected graph has at least one non-seperator vertex.
Since $\wmax\le w_j$ we know that the weight does not go below $w_j$ during the last deletion.

\paragraph*{Algorithm \algTransfer:}
\begin{enumerate}
	
	\item
Initialize $X := Q$ and let $u = \min(I(\T^-))$.
\item
For $\ell = 1$ to $|P_Q^{T_i}|$ do:
	\begin{enumerate}[label*=\arabic*.,ref=\theenumi.\arabic*]
	\item
		Let $T_j = T(P_Q^{\ell})$.
	\item
	\label{transVert::checkX}
	If $w(X) \geq w_u$: set $T_u = X$. Truncate $T_u$ if necessary and terminate the algorithm.
	\item
	\label{transVert::swallowVertices}
	If $w(X \cup T_j) \leq \UpBoundGL w_j$: update $T_j$ to $X \cup T_j$ and terminate the algorithm.
		\item
			If $T_j\notin \T^*$:
		\label{transVert::wj<wi}
		Set $T_j' = T_j \cup X$, $T_{j}=T_u$ and $T_u=T_j'$.
		Truncate $T_j$ and $T_u$ if necessary and terminate the algorithm.
		
		\item		
			\label{transVert::Case1}
			If $T_j \in \T_a^+$: divide $T_j\cup X$ into connected vertex sets $V_1, V_2$ with $w(V_1), w(V_2) \geq w_{j}$ using the construction given by \cref{lemma::divideInto2Comp}.
			Set $T_j = V_1$ and $T_u = V_2$.
			Truncate $T_j$ and $T_u$ if necessary and terminate the algorithm.
			
			\item
			\label{transVert::Case2}
			We know $T_j \in \T_b^+\cap \T^*$. Set $X = X \cup P_Q^{\ell}$ and remove $P_Q^{\ell}$ from $T_j$.		
\end{enumerate}
\end{enumerate}

%% file: GL-appendix.tex
\section{Missing details from \cref{section::BCP}}
\label{app:BCP}

\subsection{Full Proof of \cref{thm::MinMax}}
	
	Let $G=(V,E,w)$ be a vertex-weighted $K_{1,c}$-free graph. Consider DFS-tree $T_r$ of $G$ rooted at $r \in V(G)$.
	We state the following easy to see fact without proof.

\begin{lemma}
	\label{lemma:bounded_degree}
	In $T_r$, each vertex has at most $c-1$ children. 
\end{lemma}
%
  
Using this property, the following lemma finds a balanced connected vertex set, whose  removal does not disconnect the graph.   

We now prove this slight reformulation of Lemma~\ref{lemma::boundedSubgraph}.
\begin{lemma}
	Given  a $K_{1,c}$-free graph $G$ and a  DFS-tree $T_r$ of $G$. If $w(G) \geq \lambda\ge \wmax$, then there is an algorithm that finds a connected vertex set $S$ such that $\lambda \leq w(S) < (c-1) \lambda$
	and $G-S$ is connected, in $\mathcal O(|V|)$ time.
	Furthermore, the algorithm also finds a DFS-tree of $G - S$. 
\end{lemma}

\begin{proof}
	 There exist  a vertex $v$ with children $v_1,\dots, v_\ell$ such that 
	  $w(T_v) \geq \lambda$, and that $w(T_{v_i}) < \lambda$ for each $i\in [\ell]$. 
	  Such a vertex can be easily found by a bottom up traversal from the leaves in $\mathcal O(|T_v|) \subseteq \bigO(|V|)$ time.
	 If $\lambda \leq w(T_v) < (c-1) \lambda$, then we set $S=V(T_v)$ and $T_r - S$ is a spanning DFS-tree in $G - S$, and we are done.
	
	The remaining case is when $w(T_v) \geq (c-1) \lambda$ and $w(T_{v_i}) < \lambda$ for every $i \in [\ell]$.
	By \cref{lemma:bounded_degree}, we obtain $\ell \leq c-1$.
	In fact $\ell= c-1$ as otherwise $w(v)$ needs to have more than $\lambda\ge \wmax$ weight in order for $T_v$ to have $(c-1)\lambda$ weight, a contradiction.

	If $v=r$, then we choose $S = \{r\} \cup \bigcup_{i=1}^{c-2} V(T_{v_i})$.
	$S$ is a connected subgraph with 
	$w(S) = w(v) + \sum_{i=1}^{c-2} w(T_{v_i}) < \wmax + (c-2)\lambda \leq (c-1) \lambda$ and
	$w(S) = w(T_v) - w(T_{c-1}) \geq (c-1)\lambda - \lambda = (c-2) \lambda \geq \lambda$.
	Moreover, $T_{v_{c-1}}$ is a DFS-tree for $G-S$.
	
	Now, consider $v \neq r$.
	Let $u$ be the parent of $v$ in $T_r$.
	By $\ell = c-1$  we obtain that $T_r[\{u\} \cup \{v\} \cup \bigcup_{i=1}^{c-1} \{v_i\}]$ is a $K_{1,c}$.
	Consequently, there exists an edge $uv_j \in E \setminus E(T_r)$ for at least one $j \in [c-1]$, since $G$ is $K_{1,c}$-free.
	We set $S = \{v\} \cup \bigcup_{i \in [c-1] \setminus \{j\}} V(T_{v_i})$ as connected vertex set and obtain analogously to the case $v=r$ the desired weight conditions for $S$.
	Finally, we remove $S$ from $T_r$ and add the edge $uv_j$ to $T_r$ to obtain a DFS-tree of $G-S$.
\end{proof}

We now give the algorithm {\algBalancedPartition} that takes as input a connected vertex-weighted $K_{1,c}$-free graph $G(V,E,w)$ and an integer $\lambda \geq w_{\max}$.\\

\noindent
\textbf{Algorithm \algBalancedPartition:}
Compute a DFS-tree $T_r$ rooted at $r \in V$ in $G$ and initialize $\calS' = \varnothing$.
Until $w(T_r) < (c-1) \lambda$ use \cref{lemma::boundedSubgraph} to remove a connected set $S$ from $G$, add it to $\calS'$, and update $T_r$ to be the DFS-tree of $G-S$ given by \cref{lemma::boundedSubgraph}.
Finally, return $\calS = \calS' \cup \{V \setminus V(\calS')\}$.

The following lemma follows easily from the construction of Algorithm {\algBalancedPartition} and \cref{lemma::boundedSubgraph}.
\begin{lemma}
	\label{lemma::algBalPart}
	Given $G$ and $\lambda\ge \wmax$, algorithm {\algBalancedPartition} provides a {\cvp} $\calS  = \{S_1,\dots,S_m\}$ of $V$
	such that $w(S_i) \in [\lambda,(c-1)\lambda)$ for every $i \in [m-1]$ and $w(S_m) < (c-1)\lambda$ in linear time.
\end{lemma}
\begin{proof}
	Let $w(G) \geq (c-1)\lambda$.
	First of all, observe that the preconditions of \cref{lemma::boundedSubgraph} are satisfied at the first application of it.
	Generally, if an iteration is executed, then we obtain a connected subgraph $S$ with $w(S) \in [\lambda, (c-1) \lambda)$, which we add to $\calS'$, and a DFS-tree $T_r$ in $G[V \setminus V(\calS')]$. 
	Thus, the preconditions of \cref{lemma::boundedSubgraph} are still maintained after an iteration if the working tree $T_r$ weighs at least $\lambda$.
	
	We only remove vertices $S$ from $T_r$ and add $S$ to $\calS'$.
	Hence, the vertex sets in $\calS'$ are pairwise disjoint.
	We apply \cref{lemma::boundedSubgraph} until $w(T_r)$ is less than $(c-1)\lambda$.
	Since $S_m = V \setminus V(\calS') = V(T_r)$, the vertex set $S_m$ is a connected vertex set disjoint from the vertex sets in $\calS'$ and by the termination criteria $S_m$ weighs less than $(c-1)\lambda$.
	As a result, $\{S_1,\dots,S_{m}\}$ is the desired {\cvp}.
	
	Lastly, we analyze the running time.
	Computing a DFS-tree $T_r$ runs in time $\O(|V|+|E|)$.
	By starting from the leaves and using suitable data structures we can find $T_v$ according to \cref{lemma::boundedSubgraph} in $\O(|T_v|)$.
	Note that in case we add an edge $ux \in E$ as explained in the proof of \cref{lemma::boundedSubgraph}, where $u$ is the parent of $v$ and $x$ a child of $v$ in the modified $T_r$, we have already the subtree $T_x$ with its corresponding weight $w(T_x) < \lambda$ in hand, i.e.~we do not need to proceed a second time through the vertices $V(T_x)$.
	Thus, filling $\calS'$ by the resulting vertex sets from \cref{lemma::boundedSubgraph} can be performed in time $\O(|V|)$, since the vertex sets in $\calS'$ are pairwise disjoint and we see them once in the algorithm.
	As a result, the algorithm runs in time $\O(|E|)$.
\end{proof}

	\noindent
We now prove \cref{thm::MinMax}
by appropriately choosing $\lambda$ in algorithm {\algBalancedPartition}.

\begin{proof}[Proof of \cref{thm::MinMax}]
	Set $\lambda = \max\{\wmax,\frac{w(G)}{k}\}$.
	Let $X^*$ be the optimal value of Min-Max {\bcp} on $G$.
	Observe that $\lambda \leq X^*$.
	Now, apply algorithm {\algBalancedPartition} with $\lambda$ as input parameter to obtain a {\cvp} $\calS = \{S_1,\dots,S_m\}$, where $w(S_i) \in [\lambda, (c-1)\lambda)$ for $i \in [m-1]$ and $w(S_m) < (c-1)\lambda$.
	By $\lambda \leq X^*$ we obtain that $\max_{i \in [m]} w(S_i) < (c-1) \lambda \leq (c-1) X^*$.
	It remains to show that $m \leq k$.
	We have $w(G) = w(S_m) + \sum_{i=1}^{m-1} w(S_i) \geq w(S_m) + (m-1)\lambda > \frac{m-1}{k} w(G)$,
	which implies $m \leq k$.
\end{proof}

\subsection{Full Proof of \cref{thm::MaxMin}}

We also use the algorithm {\algBalancedPartition} for the max-min objective. The following property is immediate but useful.

\begin{corollary}
	\label{corollary::stillConnected}
	Let $\calS$ be the output of algorithm {\algBalancedPartition}.
	For every $j \in [|\calS|]$ the subgraph $G - \bigcup_{i=1}^j S_i$ is connected.
\end{corollary}

Let $(G,k)$ be an instance of {\maxminbcp}, where $G$ is a $K_{1,c}$-free graph.
Let $X^*$ be the optimal value for the instance $(G,k)$.
For any given  $X \leq w(G)/k$, we design an algorithm that either gives a $[\lfloor X/(c-1) \rfloor, \infty)$-{$\cvp_{k}$}, or reports that $X>X^*$.
Note that $X^*\le w(G)/k$.
Once we have this procedure in hand, a binary search for the largest $X$ in the interval $\left(0, \left\lceil w(G)/k \right\rceil \right]$ for which we find a $[\lfloor X/(c-1) \rfloor, \infty)$-{$\cvp_{k}$} can be used to obtain an approximate solution for {\maxminbcp}.

\noindent
\textbf{Algorithm \algmaxmin:}
First remove all vertices of weight more than $\lambda=\lfloor X/(c-1)\rfloor$ and save them in $H$.
Then save the connected components of weight less than $\lambda$ in $\calQ$.
Let $\calV = \{V_1,\dots,V_{\ell}\}$ be the connected components of $G - (H \cup V(\calQ))$.
Apply algorithm {\algBalancedPartition} on each $G[V_i]$ with $\lambda$ as input parameter to obtain $\calS^i = \{S^i_1, \dots, S^i_{m_i}\}$ for every $i \in [\ell]$.
If for some $i \in [\ell]$ the weight $w(S_{m_i})$ is less than $\lambda$, then merge this vertex set with $S_{m_i - 1}$ and accordingly update $\calS^i$.
Further, compute a $(\lambda, \infty)$-{$\cvp_{|H|}$} $\calS^H$ of $G[H \cup V(\calQ)]$ as follows: for each $h\in H$, we will have a set $S^h\in \calS^H$ with $h\in S^h$; we add each $Q \in \calQ$ to some $S^h$ such that $h \in N(Q)$. 
Let $\calS = \calS^H \cup \bigcup_{i=1}^{\ell} \calS^i$.
If $|\calS| \geq k$, then merge connected sets arbitrarily in $\calS$ until $|\calS| = k$ and return $\calS$.
If $|\calS|<k$, report that $X>X^*$.
\newline

We point out that a $[\lambda, \infty)$-{$\cvp_{j}$} with $j > k$, can easily be transformed to a $[\lambda, \infty)$-{$\cvp_{k}$}, since the input graph is connected.
To prove that the algorithm works correctly, we need to show that {\algmaxmin} returns the desired $[\lambda, \infty)$-{$\cvp_{k}$} if $|\calS| = k$, where $\lambda = \lfloor X/(c-1) \rfloor$.
Furthermore, we need to show that if the algorithm terminates with $|\calS| < k$ and reports $X>X^*$ that this is indeed true.
These facts will finally lead us to a successful application of a binary search and in turn to an approximate solution.

\begin{lemma}
	If algorithm {\algmaxmin} returns $\calS$, then $\calS$ is a $[\lambda, \infty)$-{$\cvp_{k}$} of $V$.
\end{lemma}
\begin{proof}
	If algorithm {\algmaxmin} returns $\calS$, then we have $|\calS| = k$.
	Therefore, we only need to show that $w(S) \geq \lambda$ and that every $S \in \calS$ is connected.
	Consider $S \in \calS^H$.
	By construction $S$ contains exactly one $h \in H$ and some $Q \in \calQ$ with $h \in N(Q)$.
	Hence, $S$ is connected.
	Moreover, $w(h) > \lambda$ for every $h \in H$ by the choice of those vertices and hence $w(S) > \lambda$.    
	
	Let $G[V_i]$ for $i \in [\ell]$ be a connected component of $G - (H \cup V(\calQ))$ with $w(V_i) \geq \lambda$.
	We applied algorithm {\algBalancedPartition} for $G[V_i]$ with $\lambda$ as input and obtained a {\cvp} $\overline{\calS} = \{S_1,\dots,S_m\}$ of $V_i$ with $w(S_i) \in [\lambda,(c-1)\lambda)$ for every $i \in [m-1]$ and $w(S_m) < (c-1)\lambda$ (cf. \cref{lemma::algBalPart}).
	Thus, $S_i \in \{S_1,\dots,S_{m-1}\}$ is connected and weighs at least $\lambda$.
	If $w(S_m) < \lambda$, then we merge $S_{m-1}$ with $S_m$, where $G[S_m \cup S_{m-1}]$ is connected, which we derive from \cref{corollary::stillConnected} as $G[S_m \cup S_{m-1}] = G[V_i] - \bigcup_{i=1}^{m-2} S_i$.
\end{proof}
For completeness, we repeat.
\begin{lemma}
	If Algorithm {\algmaxmin} terminates with $|\calS| < k$, then $X>X^*$.
\end{lemma}
\begin{proof}
	Let $H,\calQ,\calV = \{V_1,\dots,V_{\ell}\}$ be the computed vertices and connected vertex sets in the algorithm for $\lambda=\lfloor X/(c-1)\rfloor$, respectively.
	Recall that $w(V_i) \geq \lambda$ for every $V_i \in \calV$ and $w(Q) < \lambda$ for every $Q \in \calQ$.
	Let $\calS^* = \left\{S^*_1, \dots, S^*_k\right\}$ be an optimal solution of $(G,k)$, i.e., $\calS^*$ is an $[X^*, \infty)$-{$\cvp_{k}$} of $V$. 	
	Consider the sets $\calV^{H \cup \calQ} := \left\{S^*_i \in \calS^*|  S^*_i \cap (H \cup V(\calQ)) \ne \varnothing \right\}$, $\calV^{1} := \left\{S^*_i \in \calS^*|  S^*_i \cap V_1 \neq \varnothing \right\} \setminus \calV^{H \cup \calQ}$, $\dots$, $\calV^{\ell} := \left\{S^*_i \in \calS^*|  S^*_i \cap V_{\ell} \neq \varnothing \right\} \setminus \mathcal{V}^{H \cup \calQ}$.
	We claim that these sets are a partition of $\calS^*$. This follows directly from the fact that $H$ separates all $V_i \in \calV$ and all $Q \in \calQ$ from each other.
	That is, for an $i \in [k]$ and $j \in [\ell]$ the  connected vertex set $S^*_i$ with $S^*_i \cap V_j \neq \varnothing$ and $S^*_i \cap V \setminus V_j \neq \varnothing$ contains at least one $h \in H$ and hence $S^*_i \in \calV^{H \cup \calQ}$.	Otherwise, if $S^*_i \subseteq V_j$, then $S^*_i \in \calV^j$.
	
	Suppose {\algmaxmin} terminates with $|\calS|<k$ although $X\le X^*$. 	
	We show that $\left|\mathcal{V}^{H \cup \calQ}\right| \leq |H|$ and $\left| \mathcal{V}^{i}  \right| \leq |\calS^{i}|$ for every $i \in [\ell]$, implying that $\left|\calS^*\right|\leq |H| + \sum_{i=1}^{\ell} |\calS^{i}| = |\calS|<k$, which contradicts $\left|\calS^*\right|=k$.
	
	First, we show $\left|\calS^{H \cup \mathcal{Q}}\right| \leq |H|$.
	For this, it is sufficient to prove that $S^*_i \cap H \neq \varnothing$ for each $S^*_i \in \calV^{H \cup \calQ}$ as $\calS^*$ is a partition of $V$.
	We prove this by contradiction.
	Suppose there is an $S^*_i\in \calV^{H \cup \calQ}$, such that $S^*_i \cap H = \varnothing$.
	This implies that $S^*_i \subseteq Q$ for some $Q \in \calQ$, since $H$ separates every $Q \in \calQ$ from every other $Q' \in \calQ \setminus \{Q\}$ and from the vertices $V \setminus (H \cup V(\calQ))$.
	Thus, $w(S_i^*) \leq w(Q) < \lambda$ by the definition of $\calQ$ and therefore $w(S_i^*)<\lambda= \lfloor X/(c-1)\rfloor \le \lfloor X^*/(c-1)\rfloor$, contradicting $\min_{i \in [k]} w(S^*_i)=X^*$.
	
	It remains to show that $\left| \mathcal{V}^{i}  \right| \leq |\calS^{i}|$ for every $i \in [\ell]$.
	Fix an $i \in [\ell]$ and let $G[V_i]$ with $\lambda$ be the input when calling algorithm {\algBalancedPartition}.
	Observe that the input is valid, since $G[V_i]$ is connected by definition and $w(G[V_i])\geq \lambda \geq \max_{v \in V_i} w(v) $ as $H$ contains all vertices that have weight more than $\lambda$.
	Algorithm {\algBalancedPartition} provides a {\cvp} $\calS^i = \{S^i_1,\dots,S^i_{m_i}\}$ of $V_i$ with $w(S^i_j) \in [\lambda, (c-1)\lambda)$ for every $j \in [m_i - 1]$ and $w(S^i_{m_i}) < (c-1)\lambda$. 
	Consider $\calS^i$ before merging, i.e.~we do not merge $S^i_{m_i}$ to $S^i_{m_i-1}$ in the algorithm {\algmaxmin} if $w(S_{m_i}) < \lambda$.
	That is, $w(S_m) < \lambda$ is possible, and we need to show $\left| \mathcal{V}^{i}  \right| \leq m_i -1 = |\calS^{i}| -1$.
	Observe for $S^* \in \calV^{i}$ that $S^* \subseteq V_i$, i.e.~$\sum^{m_i}_{j=1} w(S^i_j) \geq \sum_{S^* \in \calV^{i}} w(S^*)$.
	As a result, we have $|\calS^{i}| X \geq |\calS^{i}| (c-1)\lambda > \sum^{m_i}_{j=1} w(S^i_j) \geq \sum_{S^* \in \calV^{i}} w(S^*) \geq |\calV^{i}|X^*$.
	Consequently, by $X \leq X^*$ we obtain  $|\calV^{i}| < |\calS^{i}|$, which leads  to $|\calV^{i}| \leq |\calS^{i}| - 1$.
\end{proof}

\paragraph*{Running time:}
Finding the heavy vertices $H$ and computing the resulting connecting components $\calV = \{V_1,\dots,V_\ell\}$ of $G-H$ can be performed in time $\O(|E|)$.
The algorithm {\algBalancedPartition} runs in time $\bigO(|E(G[V_i])|)$ by \cref{lemma::algBalPart} for every $i \in [\ell]$ and consequently, in $\O(|E|)$ in total as the vertex sets in $\calV$ are pairwise disjoint.
Hence, algorithm {\algmaxmin} runs in time $\O(|E|)$. 

We modify slightly the binary search to optimize the total running time. 
Let $g(\ell) := 2^{\ell}$, $\mathbb{N} \ni \ell \geq 1$.
We increase stepwise $\ell$ in $g(\ell)$ until we find an $\ell^*$ with $g(\ell^*) < X^* \leq \min\left(g\left(\ell^* + 1\right), w(G)/k \right) =: \widehat{X} \leq 2X^*$.
Afterwards, we perform a binary search  in the interval $X \in \left[g(\ell^*),\widehat{X}\right]$.
As a result, finding the $(c-1)$-approximate solution $\calS$ runs in time $\bigO\left(\left(\log X^* + log X^* \right)|E|\right)$ and thus, we obtain a running time in $\bigO\left(log\left(X^*\right)|E|\right)$.
This completes the proof of \cref{thm::MaxMin}.

\section{Missing details from \cref{section::GL}}
\label{app:GL}

\subsection{Preprocessing to get $\wmax<\alpha w_k$}
\label{app:preprocess}
Suppose $w_{\max} \geq \alpha w_\ell$, where $\ell$ is the smallest index in $[k]$ that satisfies this inequality.
We remove a vertex $\vmax$ with $w(\vmax) = \wmax$ from $G$ and $w_\ell$ from $\mathcal{W}$.
Further, we set $T_\ell = \{\vmax\}$ and consider the set for index $\ell$ to be finished, i.e., we now aim to find a set $\T = \{T_1, \dots, T_{\ell-1}, T_{\ell+1}, \dots, T_k\}$  according to $\mathcal{W} = \{w_1, \dots, w_{\ell-1}, w_{\ell+1}, \dots, w_k\}$.
Observe that since we only deleted one vertex, $G$ now is at least $(k-1)$-connected, and $|\mathcal{W}| = k-1$. 
Note that since $w_{\max} \leq w_k \leq w_\ell\le 3\alpha w_\ell$, we have $w(T_\ell) \in [\alpha w_\ell, 3\alpha w_\ell]$ as required.
Also, we obtain $w(G) \geq \sum_{w_j \in \mathcal{W}} w_j$ after removing $w_{\ell}$ from $\mathcal{W}$ and $\vmax$ from $G$.
After this preprocessing step, we can assume that we have a $k$-connected graph $G = (V,E,w)$ and natural numbers $w_1, \dots, w_k$ sorted in descending order, where $\sum_{i=1}^k w_i \leq w(G)$ and $\wmax < \alpha w_k$. 

\subsection{Missing lemmas for correctness of {\algBoundedGL}}
\label{app:BoundedGL}

\begin{proof}[Proof of \cref{thm:one-side} given \cref{thm::GLpartition}]
	First we prove the lower-bound version. 
	Apply \cref{thm::GLpartition} on $G$ with $\alpha = \frac{1}{3}$ and let $\T = \{T_1, \dots, T_\ell\}$ be the resulting vertex sets.
	For each $T_i\in \T$, 
	we have that $\frac{1}{3}w_i\le w(T_i)\le w_i$.
	We claim that $\ell=k$.
	Indeed, if $\ell<k$, then
	since $w(T_i) \leq w_i$ for every $i \in [\ell]$, we have that $\sum_{i=1}^\ell w(T_i)< \sum_{i=1}^k w_i = w(G)$, and hence $\T$ is not a \cvp, a contradiction to \cref{thm::GLpartition}.
	Note that each $T_i\in\T$ already satisfies the required lower bound of $w(T_i)\ge \frac{1}{3}w_i$. 
	Now, we can obtain the desired {\cvp} $\T^*$ by adding each connected component $Q$ of $G \setminus V(\T)$ to an arbitrary vertex set in $\T$ that has an edge to $Q$.
	
	Now, we prove the upper-bound version. 
	Apply \cref{thm::GLpartition} on $G$ with $\alpha = 1$ and let $\T = \{T_1, \dots, T_\ell\}$ be the resulting vertex sets.
	For each $T_i\in \T$, 
	we have that $w_i\le w(T_i)\le 3w_i$.
	We claim that $\T$ is a {\cvp} of $V$.
	Indeed, by \cref{thm::GLpartition}, if $\T$ is not a \cvp, then $\ell=k$ but then
	since $ w(T_i) \geq w_i$ for every $i \in [\ell]$, we get $\sum_{i=1}^\ell w(T_i) = w(G)$, a contradiction. 
	Note that each $T_i\in\T$ already satisfies the required upper bound of $w(T_i)\le 3w_i$. 
	If $\ell = k$, then $w(T_i) = w_i$ for every $i \in [k]$ and we are done.
	In case $\ell < k$ we remove an arbitrary vertex $v$ from a vertex set $T \in \T$ with $|T|\ge 2$ (at least one such set exists due to $k$-connectivity), such that $T$ is still a connected vertex set, and we add $\{v\}$ as a new set to $\T$.
	We repeat this till $|\T| = k$.
	Now, the resulting $\T$ is a {\cvp} with the desired upper bound conditions, since we only remove vertices from already existing sets in $\T$ and since $\wmax \leq w_k$ the new singleton vertex sets in $\T$ satisfy the required upper bound conditions. Observe that this last step does not work, if we do not require $\wmax \leq w_k$.
\end{proof}

\begin{proof}[Proof of \cref{lem:inner_loop}]
	The occurrence of one of the two cases in Step~\ref{algBoundedGL::RemoveQ} is shown in \cref{lemma::connectionQ}.
	For the correctness of Case 1, observe that we can use \cref{lemma::divideInto2Comp} to divide $T_j\cup Q$,	as $T_j\cup Q$ cannot have an $(\alpha w_j)$-separator (this would give an ($\alpha w_j)$-separator for $T_j$ implying that $T_j\in \T_b$).
	It remains to prove that the inner loop terminates as claimed.	If Step~\ref{algBoundedGL::Case1} is executed, then an $i$-big component is created in $\og$ as $\og$ now contains $V_1$ returned by the divide-routine, and hence the loop is terminated.
	The same yields if we apply the divide routine in Step~\ref{algBoundedGL::Case2}.
	So, suppose the inner loop never executes the divide routine.
	In the other cases, either a connected component is deleted from $\og$ or new vertices are added to a connected component in $\og$.
	Also note that new connected components are not introduced to $\og$ and vertices are not deleted from existing connected components (except when the whole connected component is removed; also, two or more connected components may merge due to the introduction of new vertices to $\og$).
	Thus, after $|V|^2$ iterations either there is an $i$-big component or $\og$ is empty. 
\end{proof}

\begin{lemma}
	\label{lemma::connectionQ}
In Step~\ref{algBoundedGL::RemoveQ}	at least Case 1 or Case 2 occurs.
\end{lemma}
\begin{proof}
	Suppose Case 1 does not occur i.e., $Q$ does not have an edge to any $T_j\in \T_a$.
	For $T_j\in \T_b$, let $s_j$ denote the fixed $(\alpha w_j)$-seperator vertex.
	Since $G$ is $k$-connected and $|\T|<k$, there is an edge from $Q$ to at least one vertex in $\bigcup_{T_j\in \T_b}T_j\setminus \{s_j\}$.
	Thus, Case 2 occurs.
\end{proof}

\begin{lemma}
	\label{lem:bigcomp}
	Whenever the divide routine in algorithm {\algBoundedGL} is to be executed, there is at least one $i$-big component in $\og$.
\end{lemma}
\begin{proof}
	When $i=1$ i.e.~in the first main loop iteration, this holds because $\og=G$ is connected and $\alpha w_1\leq w_1\leq \sum_{i=1}^k w_i \leq w(G)$. 
	For the subsequent iterations, the divide routine is only applied after the inner loop is terminated which only happens if either $\og$ is empty or has an $i$-big component.
	In case $\og$ is empty, then the algorithm terminates.
	So, if the algorithm applies the divide routine in the inner loop, then $\og$ has an $i$-big component.
\end{proof}
\subsection{Runtime analysis of {\algBoundedGL}}
\label{app:runtimeBGL}
\begin{lemma}
	The algorithm {\algBoundedGL} runs in $\mathcal O(k|V|^2|E|)$ time.
\end{lemma}
\begin{proof}
Clearly, the main loop has at most $k$ iterations and
the inner loop has at most $|V|^2$ iterations by \cref{lem:inner_loop}.

We categorize only the elements in $\T$ that satisfy $w(T_j) \geq 2 \alpha w_j$ into $\T_a$ and $\T_b$.
Note that $w(Q) < w_i$ as $Q$ is $i$-small.
That is, at the point where some $T_j \in \T$  exceeds $2 \alpha w_j$, its weight is less than $3 \alpha w_j$ and the category of $T_j$ can be computed.
Consequently, reaching Step~\ref{algBoundedGL::Case2} means that we have $w(T_j) \geq 2 \alpha w_j$.
Otherwise, we have $w(T_j \cup Q) \leq 3 \alpha w_j$ and would iterate with the next $i$-small connected component of $\og$.
Now, in case we remove $Q'$ from $T_j$ in Step~\ref{algBoundedGL::Case2} will show that the weight still satisfies $w(T_j \setminus Q') \geq 2 \alpha w_j$ to avoid a re-categorization of $T_j$.
This finally will ensure that we can bound the number of categorizations of $T \in \T$ into $\T_a$ and $\T_b$ throughout the whole algorithm by $2k$.
Moreover, that we can realize a check whether $T_j \cup Q$ is $\alpha w_j$-dividable in Step~\ref{algBoundedGL::Case2} in linear time.
Observe by \cref{lemma::divideInto2Comp} that a $T_j \in \T_a$ will never change its category except when we execute the divide routine on $T_j$ in Step~\ref{algBoundedGL::Case1}.
Now, we want to ensure the same for the elements in $\T_b$ in Step~\ref{algBoundedGL::Case2}.

\begin{claim}
	\label{claim::divide}
	Let $T_j \in \T_b$ with $w(T_j) \geq 2 \alpha w_j$ and $w(T_j \cup Q) \geq 3 \alpha w_j$, where $w(Q) < w_j$.
	If $G[V(C(T_j)) \cup Q]$ contains a component $R \subset T_j \cup Q$ with weight at least $\alpha w_j$, then $T_j \cup Q$ is $\alpha w_j$-dividable.
\end{claim}
\begin{proof}
	Clearly $R$ contains $Q$ as $T_j \in \T_b$ and $Q$ connects some components in $\calQ \subseteq C(T_j)$.
	As long $Q$ weigh less than $\alpha w_j$ we add individual $Q'' \in \calQ$ to $Q$.
	By $w(Q'') < \alpha w_j$ we can ensure that the resulting $Q$ weighs less than $2 \alpha w_j$ and thus $w(T_j \setminus Q) \geq \alpha w_j$ as $w(T_j \cup Q) \geq 3 \alpha w_j$.
	Since every component of $C(T_j)$ is connected to $s_j$, we produce with $Q$ and $T_j \setminus Q$ two connected vertex sets each of weight at least $\alpha w_j$.
\end{proof}

\begin{claim}
	\label{claim::2alpha}
	Let $T_j \in \T_b$ with $w(T_j) \geq 2 \alpha w_j$ and $w(T_j \cup Q) \geq 3 \alpha w_j$, where $w(Q) < w_j$.
	If $G[V(C(T_j)) \cup Q]$ contains no component with weight at least $\alpha w_j$, then in Step~\ref{algBoundedGL::Case2} we have $w(T_j \setminus Q') \geq 2 \alpha w_j$.
	Furthermore, after removing $Q'$ from $T_j$ we have $T_j \in \T_b$ with the same $\alpha w_j$-separator $s_j$.
\end{claim}
\begin{proof}
	$Q'$ is connected to $Q$ and $w(Q \cup Q') < \alpha w_j$ as $G[V(C(T_j)) \cup Q]$ contains no component with weight of at least $\alpha w_j$. Thus, $w(T_j \setminus Q') < 2 \alpha w_j$ contradicts $w(T_j \cup Q) \geq 3 \alpha w_j$.
	The second part of the claim is obviously true as we remove only $Q'$ from $C(T_j)$.
\end{proof}

As a result, once a $T \in \T$ is categorized, then a re-categorization is only necessary if we apply the divide routine on $T$.
Since this can happen once in the outer loop and this loop is executed at most $k$ times, we get at most $k$ additional re-categorizations and therefore in total at most $2k$ categorizations.
Hence, by \cref{lemma::divideInto2Comp} we may realize this in $\bigO(k|V||E|)$.

It remains to analyze the steps in the inner loop.
Applying the divide routine in 
Step~\ref{algBoundedGL::Case1} or in Step~\ref{algBoundedGL::Case2} can be performed in $\mathcal O(|V||E|)$ by \cref{lemma::divideInto2Comp}. 
This occurs at most once per iteration of the outer loop as the inner loop terminates once a divide routine is executed; note that the divide routine creates the $i$-big component $V_1$ in $\og$.
All other individual steps can be easily realized in time $\mathcal O(|E|)$, also the check, whether $T_j \cup Q$ is $\alpha w_j$-dividable in Step~\ref{algBoundedGL::Case2} (cf.~\cref{claim::divide}).
Thus the claimed runtime follows.
\end{proof}

%% file: GLDouble-app.tex
\section{Missing details from \cref{sec:double-side}}
\label{app:double-side}
First we show the existence of the required path in Step~\ref{doubleGL:findPath} of {\algDoubleBoundedGL}.
\begin{lemma}
	\label{lemma::setsOfP}
	In the algorithm {\algDoubleBoundedGL}, whenever Step~\ref{doubleGL:findPath} is reached, the sets $\calQ$ and $\T^-$ are non-empty and there is a path from $\calQ$ to $\T^-$ in $H$.
\end{lemma}
\begin{proof}
	We know that $\calQ\neq \varnothing$ due to the while loop condition.
	Note that if $|\T^+| = k$, then we have a \emph{perfect partition}, i.e.~$w(T_i) = w_i$ for every $i \in [k]$ as $\sum_{i = 1}^k w_i = w(G)$,
and hence $\calQ=\varnothing$.
Thus, we have that $|\T^+|<k$ and $\T^-\neq \varnothing$. 
	Since $|s(\T^+)| < k$, $G - s(\T^+)$ is connected.
	Hence, the transfer-graph $H$ is connected by construction.
	Since $\calQ,\T^-\neq \varnothing$, there is a path from $\calQ$ to $\T^-$ in $H$.
\end{proof}

\noindent
Now we show the correctness of {\algTransfer}.
\begin{lemma}
	\label{lemma::progress}
	After a call of algorithm {\algTransfer}, either $|\T^*|$ increases, or $|\T^*|$ remains same and the number of $\T$-assigned vertices increases.
	Furthermore, the resulting $\T$ is still pack-satisfied.
\end{lemma}
\begin{proof}
	Note that if the precondition of any of the steps~\ref{transVert::checkX} to \ref{transVert::Case1} is satisfied then the algorithm terminates.
	Suppose $\ell$ is the last iteration of the for loop.
	Then in all of the iterations $1,2,\cdots, \ell-1$,
	Step~\ref{transVert::Case2} was executed, 
	and none of the Steps~\ref{transVert::checkX} to \ref{transVert::Case1} are executed.
	Let $\T_0$ be the $\T$ input to the {\algTransfer} routine.
	Let $I^*_0=I(\T_0)$.
	We show that after the iteration $p$ for each $p\in [0,\ell-1]$, the following invariants are satisfied:
	\begin{enumerate}
		\item 	$I(\T^*)=I^*_0$, \label{inv:tstar}	
		\item $\T$ is pack-satisfied, \label{inv:pack}
		\item $X\supseteq Q$, \label{inv:XsupQ}
		\item $X\cap V(\T)=\varnothing$ ,\label{inv:XTempty}
		\item $V(\T)\cup X=V(\T_0)\cup Q$, \label{inv:XTsame}
		\item $G[X]$ is connected. \label{inv:Xconn}
		\item $X$ has an edge in $G$ to $P_Q^{p+1}$, \label{inv:XedgeP}
	\end{enumerate}

	We prove by induction on $p$.
	It is easy to see that all the invariants are satisfied for $p=0$ (i.e. before the first iteration).
	Now, we do the induction step.
	In the iteration $p$, we remove $P_Q^{p}$ from $T_j$.
	Since $w(P_Q^{p})\le w_j$, and $w(T_j\cup X)\ge 3w_j$ (otherwise precondition of Step~\ref{transVert::swallowVertices} would have been satisfied) and $w(X)\le w_u\le w_j$ (if $w(X)\ge w_u$ then precondition of Step~\ref{transVert::checkX} would have been satisfied; the inequality $w_u\le w_j$ is because $T_j \in \T^*$), we have that $T_j$ still has at least $w_j$ weight after the iteration.
	Also, $w(T_j)\le\UpBoundGL w_j$ after truncation.
	Thus $T_j\in \T^+$ and it follows that invariants \ref{inv:tstar}, \ref{inv:pack} are satisfied.
	Since we only add vertices to $X$, invariant~\ref{inv:XsupQ} is satisfied.
	Invariant~\ref{inv:XTempty} and \ref{inv:XTsame} are satisfied because of the corresponding induction assumptions and that the vertex set newly added to $X$ is removed from $T_j$ (and hence from $V(\T))$.
	Invariant~\ref{inv:Xconn} follows as $X$ was connected before (by induction assumption), the newly added set $P_Q^{p}$ is connected, and there is an edge from $X$ to $P_Q^{p}$ (by induction assumption). 
	Since $P_Q^{p}$ has an edge to $P_Q^{p+1}$, invariant~\ref{inv:XedgeP} follows.

	Now, we analyze the final iteration $\ell$.
	In this iteration, clearly one of the steps \ref{transVert::checkX} to \ref{transVert::Case1} satisfies the precondition and the algorithm terminates after this step.
	Note that if $\ell=|P_Q^{T_i}|$, then the precondition of Step~\ref{transVert::wj<wi} is satisfied (if that step is reached without termination).
	We branch on which of the terminating steps is executed. 
	For each of the steps \ref{transVert::checkX} to \ref{transVert::Case1} we show that during the step:
	\begin{enumerate}[label=(\alph*)]
		\item each $T_j\in \T^*$ remains in $\T^+$	\label{term:starplus}
		\item either $T_u$ moves from $\T^-$ to $\T^+$ or the vertices that were in $V(\T)\cup X$ before the step ends up in $V(\T)$ after the step.
			\label{term:TuOrAssign}
		\item $\T$ remains pack-satisfied.\label{term:pack}
	\end{enumerate}

	The above three statements are sufficient to prove the lemma because:
	from \ref{term:starplus} and \ref{term:TuOrAssign}, we have that either $|\T^*|$  increases or $|\T^*|$ remains same and all the vertices in $V(\T)\cup X$ becomes $\T$-assigned, during the terminating step.
	Recall that by invariant~\ref{inv:tstar}, we have that $\T^*$ remains the same during previous iterations and that $V(\T)\cup X=V(\T_0)\cup Q$ before the terminating step.
	Thus, we get that during the {\algTransfer} routine, either $\T^*$ increases, or $|\T^*|$ remains same and the number of $\T$-assigned vertices increase.
	From \ref{term:pack}, we get the pack-satisfiability of $\T$.
	
	It only remains to show that for each of the steps \ref{transVert::checkX} to \ref{transVert::Case1}, \ref{term:starplus}, \ref{term:TuOrAssign}, and \ref{term:pack} are satisfied.
	\begin{description}
		\item[Step~\ref{transVert::checkX}:]
			We have $w(X) \geq w_u$. 
			Hence, setting $T_u = X$ and truncating $T_u$ ensures that $w_u\le T_u\le \UpBoundGL w_u$. 
			Since $X$ is connected by invariant~\ref{inv:Xconn}, 
we have that $T_u$ is connected.  
			Thus $T_u$ moves from $\T^-$ to $\T^+$. Thus \ref{term:TuOrAssign} is satisfied. 
			Since the only set in $\T$ that is modified is $T_u$, \ref{term:starplus} and \ref{term:pack} are also satisfied.
		
		\item[Step~\ref{transVert::swallowVertices}:]
		Observe that $X$ has an edge to $T_j$ by invariant~\ref{inv:XedgeP}, and hence $X\cup T_j$ is connected.
		We have $w(T_j \cup X) \leq \UpBoundGL w_i$ by the precondition of the step. 
		Also, $w(T_j\cup X)\ge w(T_j)\ge w_j$ where the latter inequality follows from that $\T$ was pack-satisfied.
		Since no other set in $\T$ is modified, we have that \ref{term:starplus} and \ref{term:pack} are satisfied.
		Also, \ref{term:TuOrAssign} is satisfied as all vertices that were in $X$ are now $\T$-assigned.
		\item[Step~\ref{transVert::wj<wi}:]
		Observe that $X$ has an edge to $T_j$ by invariant~\ref{inv:XedgeP}, and hence $X\cup T_j$ is connected.
		We have $w(T_j')=w(X \cup T_j) > \UpBoundGL w_j\ge w_u$ where the former inequality is by using that precondition of Step~\ref{transVert::swallowVertices} is not satisfied and the latter is because $\UpBoundGL w_j \ge w_1\ge w_u$.
			Thus $w(T_u)$ after the step becomes greater than $w_u$, thus moving $T_u$ to $\T^+$.
			The only other set in $\T$ that is modified is $T_j$.
			We know 
			$w_j < w_u$ as $T_j$ was not in $\T^*$ (using the definition of $u$ and $\T^*$).
			Thus, before the step we had $w(T_u) \geq \frac{1}{3} w_u\ge \frac{1}{3}w_j$, where the first inequality uses that $\T$ was pack-satisfied by invariant~\ref{inv:pack}.
			This implies that after the step $w(T_j)\ge \frac{1}{3}w_j$.
			Also, $T_j$ is connected after the step as $T_u$ was connected before the step.
		Truncating $T_j$ and $T_u$ if necessary leads to $\T$ being pack-satisfied.
		
		\item[Step~\ref{transVert::Case1}:]
			Note that we have $w_j\ge w_u$ as $T_j\in \T^*$ (otherwise the Step~\ref{transVert::wj<wi} would have been executed).
		The only sets of $\T$ that are modified in this step are $T_j$ and $T_u$.
		Since $V_1$ and $V_2$ are connected and $w(V_1), w(V_2) \geq w_{j}\ge w_u$, we have that 
		$T_j$ and $T_u$ are each connected and has weight at least $w_j$ and $w_u$ respectively.
		Thus $T_j$ remains in $\T^+$ and $T_u$ moves to $\T^+$.
		Truncating $T_j$ and $T_u$ if necessary leads to $\T$ being still pack-satisfied.\qedhere
	\end{description}
\end{proof}

\noindent
Now, we show the correctness of {\algDoubleBoundedGL}.
\begin{lemma}
	\label{lemma::DoubleGLIsCorrect}
	At the end of {\algDoubleBoundedGL}, the packing $\T$ is a CVP$_k$ with the weight bounds as required by \cref{thm:both-side}.
\end{lemma}
\begin{proof}
For this, first observe that after at most $k|V|$ iterations of the while-loop $V(\T)=V(G)$.
This is because we guarantee that in each iteration, either $|\T^*|$ increases, or 
	$|\T^*|$ remains same and
	the number of $\T$-assigned vertices increase.
	So, the while loop terminates after $k|V|$ iterations and hence the algorithm terminates with $V(\T)=V(G)$. 
	Since we maintained that $\T$ is pack-satisfied, we get that $\T$ satisfies the required conditions for the partition required by \cref{thm:both-side}.
\end{proof}
Finally, we analyze the total running time of algorithm {\algDoubleBoundedGL}.

\paragraph*{Running time:}
We show that the algorithm {\algDoubleBoundedGL} runs in time $\bigO(k|V|^2|E|)$.
Step~\ref{doubleGL:lowerBound}~of algorithm {\algDoubleBoundedGL} runs in time $\bigO(k|V|^2|E|)$ by \cref{thm:one-side}.
The dominating steps in the while loop of Step~\ref{doubleGL:while} are the categorization of sets in $\T$ into $\T_a^+$ and $\T_b^+$ (including finding the seperator $s(T_i)$) and the Step~\ref{doubleGL::transStep} that applies algorithm {\algTransfer}.
The former can be done in time $\bigO(|V||E|)$ by \cref{lemma::divideInto2Comp}.

In algorithm {\algTransfer} we invoke the divide algorithm of Lemma~\ref{lemma::divideInto2Comp} at most once as the algorithm terminates in that case (Step~\ref{transVert::Case1}). 
This costs $\bigO(|V||E|)$ time by \ref{lemma::divideInto2Comp}.
It is easy to see that each of the other steps in the for loop of algorithm {\algTransfer} needs only $\bigO(|E|)$ time.
Since $|P_Q^{T_i}|\le |V|$, the number of iterations of for loop is at most $|V|$.
Thus, algorithm {\algTransfer} runs in time $\bigO(|V||E|)$.

The while-loop of {\algDoubleBoundedGL} iterates at most $k|V|$ times by \cref{lemma::DoubleGLIsCorrect}.
As a result, algorithm {\algDoubleBoundedGL} runs in time $\bigO(k|V|^2|E|)$.